\documentclass[11pt, notitlepage]{article}
\usepackage{amssymb,amsmath,comment}
\catcode`\@=11 \@addtoreset{equation}{section}
\def\thesection{\arabic{section}}

\def\theequation{\thesection.\arabic{equation}}
\catcode`\@=12
\usepackage{colortbl}
\usepackage[mathscr]{eucal}
\usepackage{epsf}
\usepackage{a4wide}
\def\R{\mathbb{R}}

\newcommand{\e}{\epsilon}

\newcommand{\al} {\alpha}

\newcommand{\Om} {\Omega}

\newcommand{\De} {\Delta}
\newcommand{\la} {\lambda}

\newcommand{\noi} {\noindent}

\newcommand{\oline} {\overline}
\newcommand{\mb} {\mathbb}
\newcommand{\mc} {\mathcal}

\usepackage[all]{xy}
\catcode`\@=11
\def\theequation{\@arabic{\c@section}.\@arabic{\c@equation}}
\catcode`\@=12

\def\N{{I\!\!N}}
\newcommand{\QED}{\rule{2mm}{2mm}}

\newtheorem{Theorem}{Theorem}[section]
\newtheorem{Lemma}[Theorem]{Lemma}
\newtheorem{Proposition}[Theorem]{Proposition}

\newtheorem{Remark}[Theorem]{Remark}
\newtheorem{Definition}[Theorem]{Definition}
\newtheorem{Example}{Example}

\begin{document}

{\vspace{0.01in}}

\title
{ \sc n-Kirchhoff Choquard equations with exponential nonlinearity}

\author{~~R. Arora,~~J. Giacomoni\footnote{LMAP (UMR E2S-UPPA CNRS 5142) Bat. IPRA, Avenue de l'Universit\'e F-64013 Pau, France. email: rakesh.arora@univ-pau.fr, jacques.giacomoni@univ-pau.fr},~~ T. Mukherjee\footnote{Tata Institute of Fundamental Research(TIFR) Centre of Applicable Mathematics, Banglore, India. e-mail: tulimukh@gmail.com}~ and ~K. Sreenadh\footnote{Department of Mathematics, Indian Institute of Technology Delhi, Hauz Khaz, New Delhi-110016, India. e-mail: sreenadh@gmail.com} }

\date{}

\maketitle

\begin{abstract}
\noi This article deals with the study of the following Kirchhoff equation with exponential nonlinearity of Choquard type (see $(KC)$ below). We use the variational method in the light of Moser-Trudinger inequality to show the existence of weak solutions to $(KC)$. Moreover, analyzing the fibering maps and minimizing the energy functional over suitable subsets of the Nehari manifold, we prove existence and multiplicity of weak solutions to convex-concave problem $(\mathcal{P}_{\la,M})$  below.

\noi \textbf{Key words:} Doubly non local equation, Kirchhoff equation, Choquard nonlinearity with critical growth, Moser-Trudinger inequality, Nehari Manifold.

\noi \textit{2010 Mathematics Subject Classification:} 35R11, 35R09, 35A15.

\end{abstract}

\section{Introduction} This article is concerned with the study of the following Kirchhoff equation with exponential nonlinearity of Choquard type
\begin{equation*}
     (KC)
     \left\{
         \begin{alignedat}{2}
             {} -m(\int_\Om|\nabla u|^n~dx)\De_n u
             & {}=  \left(\int_{\Om} \frac{F(y,u)}{|x-y|^\mu}dy\right)f(x,u),\; u>0
             && \quad\mbox{ in }\, \Omega ,
             \\
             u & {}= 0
             && \quad\mbox{ on }\, \partial\Omega,
          \end{alignedat}
     \right.
\end{equation*}
where $-\Delta_n u =-div (|\nabla u|^{n-2}\nabla u)$, $\mu\in (0,n)$, $\Om$ is a smooth bounded domain in $\mb R^n$, $n \geq 2$, $m :\mb R^+ \to \mb R^+$ and $f: \Omega \times \mb R \to \mb R$ are continuous functions satisfying suitable assumptions specified in detail, below. The function $F$ denotes the primitive of $f$ with respect to the second variable.
We also study the existence and multiplicity of solution of the following  Kirchhoff equation with convex-concave nonlinearity
 \begin{equation*}
     (\mathcal{P}_{\la, M})
     \left\{
         \begin{alignedat}{2}
             {} -m\left(\int_\Om|\nabla u|^n~dx\right)\De_n u
             & {}= (|x|^{-\mu}\ast F(u))f(u) + \lambda h(x)|u|^{q-1} u
             && \quad\mbox{ in }\, \Omega,
             \\
             u & {}= 0
             && \quad\mbox{ on }\, \partial\Omega,
             \\
             \quad u & > 0 &&\quad\mbox{ in }\, \Omega
          \end{alignedat}
     \right.
\end{equation*}
where $\mu\in (0,n)$, $\Om$ is a smooth bounded domain in $\mb R^n$, $f(u)=u|u|^p exp(|u|^{\beta})$, $0 < q<n-1 < 2n-1 < p+1= \beta_0 + (n-1)$, $\beta \in \bigg(1, \frac{n}{n-1}\bigg)$ and $F(t)=\int_{0}^t f(s)~ds$. We assume $m(t)=at+b$ where $a, b>0$ and $h \in L^{r}(\Omega)$, with $r =\frac{n}{n-q-1}$, satisfying $h^+ \not\equiv 0$.\\
The main feature of these kind of problems is its doubly-nonlocal structure due to the presence of non-local kirchhoff and Choquard term which makes the equation $(KC)$ and $(\mathcal{P}_{\la, M})$ no longer a pointwise identity. The doubly non-local nature induces some further mathematical difficulties in the use of classical methods of nonlinear analysis.\\

\noi The study of elliptic equations with nonlinearity having critical exponential growth is related to the following Trudinger-Moser inequatlity proved in \cite{Trud-Moser}:
\begin{Theorem}\label{TM-ineq}
For $n\geq 2$, $u \in W^{1,n}_0(\Om)$
\[\sup_{\|u\|\leq 1}\int_\Om \exp(\alpha|u|^{\frac{n}{n-1}})~dx < \infty\]
if and only if $\alpha \leq \alpha_n$, where $\alpha_n = n\omega_{n-1}^{\frac{1}{n-1}}$ and $\omega_{n-1}=$ $(n-1)-$ dimensional surface area of $\mb S^{n-1}$.
\end{Theorem}
The embedding $W^{1,n}_0(\Om)\ni u \mapsto \exp(|u|^\beta)\in L^1(\Om)$ is compact for all $\beta \in \left[1, \frac{n}{n-1}\right)$ and is continuous for $\beta = \frac{n}{n-1}$. Consequently the map $T: W^{1,n}_0(\Om) \to L^q(\Om)$, for $q \in [1,\infty)$, defined by $T(u):= \exp\left( |u|^{\frac{n}{n-1}}\right)$ is continuous with respect to the norm topology.
%This Theorem motivates the notion of criticality for any function $f :\Omega \times \mb R \to \mb R$, in the space $W^{1,n}_0(\Om)$, as follows
%\[\lim_{|u|\to \infty}\frac{f(x,u)}{\exp(\alpha |u|^{\frac{n}{n-1}})}=0, \;\; \forall \alpha >\alpha_0 \; \text{and}\; \lim_{|u|\to \infty}\frac{f(x,u)}{\exp(\alpha |u|^{\frac{n}{n-1}})}=+ \infty,\;\; \forall \alpha <\alpha_0 \]
%uniformly in $x \in \Om$, where $\alpha_0$ is some fixed real number.\\

\noi The study of Kirchhoff problems was initiated in 1883, when Kirchhoff \cite{kir} studied the following equation
\[\rho \frac{\partial^2u}{\partial t^2} -\left( \frac{P_0}{h} + \frac{E}{2L} \int_0^L\left|\frac{\partial u}{\partial x}\right|^2~dx \right)\frac{\partial^2u}{\partial x^2}=0,\]
where $\rho, P_0, h, E, L$ represents physical quantities. This model extends the classical D'Alembert wave equation by considering the effects of the changes in the length of the strings during the vibrations. More general versions of these problems are termed as the Kirchhoff equations and has been extensively studied by researchers till date. Such equations also appear in biological systems where the function $u$ describes a phenomenon which depends on the average of itself (such as population density), refer \cite{alv1,alv2} and references therein. We cite \cite{cheng,fig1,fig2,LLS,wang,LLG} as references where the Kirchhoff equations have been treated by variational methods, with no attempt to provide the complete list.\\

\noi On a similar note, recently, researchers are paying a lot of attention on nonlocal problems involving the nonlinearity of convolution type. They are termed as Hartree type or the Choquard type nonlinearity. Consider the problem
\[(C):\;\;\;  -\Delta u + V(x)u = (|x|^{-\mu}\ast F(x,u))f(x,u)\; \text{in}\; \mb R^n\]
where $\mu \in (0,n)$, $F$ is the primitive of $f$ with respect to second variable, $V$, $f$ are continuous functions satisfying certain assumptions. The starting point of studying such problems was the work of S. Pekar (see \cite{pekar}) in $1954$ where he used such equation to describe the quantum theory of a polaron at rest and later, P. Choquard (see \cite{choq}) in $1976$ used it to model an electron trapped in its own hole while a certain approximation to Hartree-Fock theory of component plasma is performed. The problem $(C)$ also appears when we look for standing waves of the nonlinear nonlocal Schr\"{o}dinger equation which is known to influence the propagation of electromagnetic waves in plasma \cite{berge}. Moreover, such problems play a key role in the Bose-Einstein condensation, refer \cite{BE}.
For interested readers, we refer the survey paper on Choquard equations by Moroz and Schaftingen \cite{MS}. In $2015$, L\"{u} \cite{lu} studied the following Choquard equation involving Kirchhoff operator
\[-\left( a+ b\int_{\mb R^3}|\nabla u|^2~dx \right)\Delta u + (1+ \mu g(x))u = \left(|x|^{-\alpha}\ast |u|^p \right)|u|^{p-2}u \; \text{in}\; \mb R^3\]
where $a > 0,\ b \geq 0$ are constants, $\al \in (0, 3), p \in (2, 6 - \alpha),\mu > 0$ is a parameter and $g$ is a nonnegative continuous potential satisfying some conditions. By using the Nehari manifold and the concentration compactness principle, he establishes the existence of ground state solutions when $\mu$ is large enough and studies the concentration behavior of these solutions as $ \mu \to +\infty$. Recently, Li, Gao and Zu \cite{LGZ} studied the existence and the concentration of sign-changing solutions to a class of Kirchhoff-type systems with Hartree-type nonlinearity in $\mb R^3$ using minimization argument on the sign-changing Nehari manifold and a quantitative deformation lemma. Pucci et al.\ \cite{pucci} also studied existence of nonnegative solutions of a Schr\"odinger-Choquard-Kirchhoff type fractional $p$-equation via variational methods.\\

\noi An important question now arises is the case of critical dimension $n=2$. But there is not much literature concerning problem $(C)$ when $n=2$ except the articles by Alves et al. \cite{yang-JDE,yang-JCA}. In \cite{yang-JDE}, authors studied a singularly perturbed nonlocal Schr\"odinger equation using variational methods. We point out that there is no work on Kirchhoff equations involving Choquard equations when $n=2$ till date. So our work is new in this regard where we have considered the problem with a more general quasilinear elliptic operator, the $n$-laplace operator, in the dimension $n\geq 2$. As pointed out in the beginning, the critical growth of the nonlinearity in this case is of exponential type, motivated by the Trudinger-Moser inequality. The problem of the type $(KC)$ for $n=2$ without the convolution term that is
\[-m(\int_\Om |\nabla u|^2)\Delta u ~dx = f(x,u)\; \text{in}\Om,\;\; u = 0 \;\text{on}\; \partial\Om\]
was studied by Figueiredo and Severo \cite{FS}. This result was later  extended for the $n$-Laplace operator by Goyal et al. in \cite{SPS}. It is then a natural question to investigate the existence results for a Kirchhoff equation involving Choquard nonlinearity with exponential growth.\\
%\textcolor{red}{The starting point of the study of quasilinear problems involving exponential type nonlinearity, over bounded domains, was \cite{adi}. The existnce and multiplicity results for these kind of problems with exponential critical growth was initiated and studied by Adimurthi \cite{adi} and \cite{adi-yadava}. Starting with the pioneering work of Tarantello \cite{taran}, Ambrosetti-Brezis-Cerami \cite{ambrce}, a lot of work has been done in the direction of multiplicity of positive solutions for semilinear and quasilinear problems with positive nonlinearlities. During the last decade, several authors (see ref. \cite{alves-hamidi}, \cite{bro-wu}, \cite{dra-poho}, \cite{hamidi}, \cite{wu1}-\cite{wu3} ) used the Nehari manifold and associated fiber maps approach to study the multiplicity results with polynomial type non-linearity with sign changing weight functions. In \cite{sar-paw-sree}, authors studied the existence and multiplicity of quasilinear problem with exponential nonlinearlities with non-local term Kirchhoff type by using Nehari Manifold and fibering map analysis. Recently, in \cite{puccixiangzhang} and \cite{lu}, authors studied the existence results of Schrodinger-Choquard-Kirchhoff equations involving the fractional p-Laplacian.\\

 \noi Precisely, in the first part of the present work, we prove Adimurthi \cite{adi} type existence result for the n-Kirchhoff Choquard problem $(KC)$ with nonlinearity $f(x,u)$ that has exponential critical growth and superlinear behavior at $0$. The nonlinear nature of the second order operator $-\Delta_n$ requires to show the pointwise convergence of gradients for the Palais-Smale sequences. For that, we analyze the occurrence of concentration phenomena for any Palais Smale sequence associated to $(KC)$. This concentration compactness analysis is further used to establish the Palais Smale condition for Palais Smale sequences whose energy levels are strictly below some determined critical level. Due to the doubly nonlocal feature given by the interaction between the Kirchhoff and Choquard term, the task appeal new non trivial estimates with the help of the semigroup property of the Riesz potential and the Lion's compactness Lemma (see Lemma~\ref{Lions-lem}). Since the energy functional posseses the Moutain pass geometry, we are then able to prove the existence of a Palais Smale sequence with subcritical energy level and consequently the existence of at least one solution to $(KC)$. For that we need crucially that the nonlinearity satisfies a growth condition given by \eqref{h-growth} (see Lemma~\ref{PS-level}).\\

\noi Next question that arises is the multiplicity of such Kirchhoff-Choquard equations with exponential nonlinearities. So in the second part of the present work, we study the existence and multiplicity results for problems with an extra $n$-sublinear  sign changing term by using the Nehari manifold techniques.  Precisely, we study $(\mathcal{P}_{\la,M})$ to obtain in the subcritical case ($\beta<\frac{n}{n-1}$) the multiplicity of the solutions with respect to the parameter $\lambda$ by extracting the Palais Smale sequence in the natural decomposition of the Nehari Manifold.  This requires very accurate estimates on the energy functional restricted to the components of the Nehari manifold. In the critical case ($\beta=\frac{n}{n-1}$), we use again the concentration compactness together with an accurate analysis of the level of the energy functional on the Nehari manifold to determine potential concentration phenomena for associated Palais Smale sequences. Based on this analysis, we show for $\lambda$ small enough the existence of a relatively compact Palais Smale sequence that yields at least one  solution to $(\mathcal{P}_{\la,M})$.\\

\noi During  last few decades, several authors such as in \cite{alves-hamidi,bro-wu,dra-poho,hamidi,wu1,wu2,wu3} used the Nehari manifold and associated fiber maps approach to study the multiplicity results with polynomial type nonlinearity and sign changing weight functions whereas the n-Laplace problems with exponential type nonlinearity has been addressed in \cite{SPS,goyal1,goyal2}. In case of Kirchhoff equations with Choquard nonlinearity, we highlight that no result is avalaible in the current literature. In this regard, the results proved in the present paper are completely new.

\section{Main results}
First, we consider the problem $(KC)$.
%\begin{equation*}
   %  (KC)
    % \left\{
       %  \begin{alignedat}{2}
           %  {} -m(\int_\Om|\nabla u|^n~dx)\De_n u
            % & {}=  \left(\int_{\Om} \frac{F(y,u)}{|x-y|^\mu}dy\right)f(x,u)
             %&& \quad\mbox{ in }\, \Omega ,
             %\\
            % u & {}= 0
            % && \quad\mbox{ on }\, \partial\Omega ,
          %   \\
            % u & {} > 0
           %  && \quad  \mbox{ in }\, \Omega,
         % \end{alignedat}
    % \right.
%\end{equation*}%
%where $\Om$ is smooth bounded domain in $\mb R^n$.
The function $m : \mb R^+ \to \mb R^+$ is a continuous function satisfying the following conditions:
\begin{enumerate}
\item[(m1)] There exists $m_0>0$ such that $m(t)\geq m_0$ for all $t\geq 0$ and $M(t)=\displaystyle\int_0^t m(s)ds$ satisfies
\[M(t+s)\geq M(t)+M(s), \; \text{for all}\; t,s\geq 0.\]
\item[(m2)] There exists constants $b_1,b_2>0$ and $\hat t>0$ such that for some $r\in \mb R$
\[m(t)\leq b_1+b_2t^r,\;\text{for all}\; t \geq \hat t.\]
\item[(m3)] The function $\frac{m(t)}{t}$ is non-increasing for $t>0$.
\end{enumerate}
\begin{Example}
An example of a function satisfying (m1), (m2) and (m3) is $m(t)= m_0+ bt^\beta$ where $m_0,\beta<1$ and $b\geq0$. Also $m(t)= 1 +\log(1+t)$ for $t\geq 1$ verifies (m1)-(m3).
\end{Example}
Using (m3), one can easily deduce that the function
\[(m3)^\prime \quad \quad \quad \frac{1}{n}M(t)-\frac{1}{\theta}m(t)t \;\text{is non-negative and non-decreasing for}\; t\geq 0 \;\text{and}\; \theta \geq 2n.\]
The function $f:\Om \times \mb R\to \mb R$ is given by $f(x,t)=h(x,t)\exp(|t|^{\frac{n}{n-1}})$. In the frame of problem $(KC)$, $h \in C(\bar \Om \times \mb R)$ satisfies the following conditions
\begin{itemize}
\item[(h1)] $h(x,0)=0$ for all $t \leq 0$ and $h(x,t)>0$ for $t>0$.
 \item[(h2)]  For any $\e>0$, $\lim\limits_{t \to \infty}\sup_{x \in \bar \Om}h(x,t)\exp(-\e|t|^{\frac{n}{n-1}})=0$ and $\lim\limits_{t \to \infty}\inf_{x \in \bar \Om}h(x,t)\exp(\e|t|^{\frac{n}{n-1}})=\infty$.
    \item[(h3)] There exists $\ell >n-1$ such that $\frac{h(x,t)}{t^{\ell}}$ is increasing for each $t>0$ uniformly in $x \in \Om.$
        \item[(h4)] There exist $T, T_0>0$ and $\gamma_0 >0 $ such that $0<t^{\gamma_0}F(x,t)\leq T_0 f(x,t)$ for all $|t|\geq T$ and uniformly in $x \in \Omega$.
\end{itemize}
%Is is easy to see that (h3) implies that for each $K_1>0$ it holds
%\[(h3)^\prime \quad \quad \quad K_1 F(x,t)\leq tf(x,t), \; \text{for all}\; (x,t)\in \Om \times \mb R. \]
The condition (h3) implies that $\frac{f(x,t)}{t^{n-1}}$ is increasing for each $t>0$ uniformly in $x \in \Om$. \\
\begin{Example}
An example of functions satisfying $(h1)-(h4)$ is $f(x,t)= t^{\beta_0+(n-1)}\exp(t^p)\exp(|t|^{\frac{n}{n-1}})$ for $t \geq 0$ and $f(x,t)=0$ for $t<0$  where $0\leq p< \frac{n}{n-1}$ and $\beta_0>0$.
\end{Example}
\begin{Definition}
We call a function $u \in W^{1,n}_0(\Om)$ to be a solution of $(KC)$ if
\[m(\|u\|^n) \int_\Om |\nabla u|^{n-2}\nabla u. \nabla \varphi ~dx = \int_\Om \left(\int_{\Om}\frac{F(y,u)}{|x-y|^{\mu}}dy\right)f(x,u)\varphi ~dx,\; \text{for all}\; \varphi \in W^{1,n}_0(\Om).\]
\end{Definition}
The energy functional $E : W^{1,n}_0(\Om) \to \mb R$ associated to $(KC)$ is given by
\[E(u)= \frac{1}{n}M(\|u\|^n) - \frac12 \int_\Om \left(\int_{\Omega} \frac{F(y,u)}{|x-y|^\mu}dy\right)F(x,u)~dx.\]
Under the assumptions on $f$, we get that for any $\e>0$, $p \geq 1$ and $0\leq \beta_0<\ell$, there exists $C(\e,n,\mu)>0$ such that for each $x \in \Om$
\begin{equation}\label{kc-1}
|F(x,t)|\leq \e |t|^{{\beta_0+1}}+ C(\e,n,\mu)|t|^p\exp((1+\e)|t|^{\frac{n}{n-1}}), \; \text{for all}\; t \in\mb R.
\end{equation}
For any $u \in W^{1,n}_0(\Om)$, by virtue of Sobolev embedding we get that $u \in L^q(\Om)$ for all $q \in [1,\infty)$. %Since $\mu \in (0,n)$, using H\"{o}lder's inequality, we get that
%\[\int_\Om |u|^{\frac{2pn}{2n-\mu}}\exp\left(\frac{2pn}{2n-\mu}|u|^{\frac{n}{n-1}}\right)~dx \leq \left(\int_\Om |u|^{\frac{2n\alpha^\prime}{2n-\mu}} \right)^{1/\alpha^\prime} \left(\exp\left(\frac{2n\alpha}{2n-\mu}|u|^{\frac{n}{n-1}}\right) \right)^{1/\alpha}< +\infty\]
%where $\alpha> 1$.
This also implies that
\begin{equation}\label{KC-new1}
F(x,u) \in L^{q}(\Om)\mbox{ for any }q\geq 1.
\end{equation}
Now we recall the well known Hardy-Littlewood-Sobolev inequality.
 \begin{Proposition}\label{HLS}
(\textbf {Hardy-Littlewood-Sobolev inequality}) {[pp. 106, Theorem 4.3, \cite{lieb}]} Let $t,r>1$ and $0<\mu<n $ with $1/t+\mu/n+1/r=2$, $f \in L^t(\mathbb R^n)$ and $h \in L^r(\mathbb R^n)$. There exists a sharp constant $C(t,n,\mu,r)$, independent of $f,h$ such that
 \begin{equation}\label{HLSineq}
 \int_{\mb R^n}\int_{\mb R^n} \frac{f(x)h(y)}{|x-y|^{\mu}}\mathrm{d}x\mathrm{d}y \leq C(t,n,\mu,r)\|f\|_{L^t(\mb R^n)}\|h\|_{L^r(\mb R^n)}.
 \end{equation}
{ If $t =r = \textstyle\frac{2n}{2n-\mu}$ then
 \[C(t,n,\mu,r)= C(n,\mu)= \pi^{\frac{\mu}{2}} \frac{\Gamma\left(\frac{n}{2}-\frac{\mu}{2}\right)}{\Gamma\left(n-\frac{\mu}{2}\right)} \left\{ \frac{\Gamma\left(\frac{n}{2}\right)}{\Gamma(n)} \right\}^{-1+\frac{\mu}{n}}.  \]
 In this case there is equality in \eqref{HLSineq} if and only if $f\equiv (constant)h$ and
 \[h(x)= A(\gamma^2+ |x-a|^2)^{\frac{-(2n-\mu)}{2}}\]
 for some $A \in \mathbb C$, $0 \neq \gamma \in \mathbb R$ and $a \in \mathbb R^n$.}
 \end{Proposition}

Taking $t=r=\frac{2n}{2n-\mu}$ in Proposition \ref{HLS} and using \eqref{KC-new1}, we get that $E$ is well defined. Also $E \in C^1(W^{1,n}_0(\Om), \mb R)$. Naturally, the critical points of $E$ corresponds to weak solutions of $(KC)$ and for any $u\in W^{1,n}_0(\Om)$ we have
\[\langle E^\prime(u),\varphi\rangle = m(\|u\|^n) \int_\Om |\nabla u|^{n-2}\nabla u\nabla \varphi~dx - \int_\Om \left(\int_\Om \frac{F(y,u)}{|x-y|^{\mu}}dy\right)f(x,u)\varphi~dx \]
for all $\varphi \in W^{1,n}_0(\Om)$. The following theorem is the main result concerning $(KC)$ proved in this article.

\begin{Theorem}\label{kc-mt-1}
Assume (m1)-(m3) and (h1)-(h4) holds. Assume in addition
\begin{equation}\label{h-growth}
\displaystyle \lim_{s\to +\infty} \frac{sf(x,s)F(x,s)}{\exp\left(2 |s|^{\frac{n}{n-1}}\right)} = \infty,\mbox{ uniformly in }x \in \overline{\Om}.
\end{equation}
 Then the problem $(KC)$ admits a weak solution.
\end{Theorem}
\begin{Example}
An example of function $f$ satisfying (h1)-(h4) and \eqref{h-growth} is $f(x,t)=g(x)t^{p}\exp({t^{\frac{n}{n-1}}})$ for $t\geq 0$ and with $0\not\equiv g\in L^\infty(\Om)$ and non negative and $p>n-1$.
\end{Example}
 We also study the existence of positive solutions to  the  perturbed quasilinear Kirchhoff equation $(\mathcal{P}_{\la, M})$.
% \begin{equation}\label{prob}
   %  (\mathcal{P}_{\la, M})
   %  \left\{
      %   \begin{alignedat}{2}
          %   {} -m\left(\int_\Om|\nabla u|^n~dx\right)\De_n u
           %  & {}= (|x|^{-\mu}\ast F(u))f(u) + \lambda h(x)|u|^{q-1} u
           %  && \quad\mbox{ in }\, \Omega,
           %  \\
           %  u & {}= 0
           %  && \quad\mbox{ in }\, \partial\Omega ,
            % \\
          %   \quad u & > 0 \mbox{ in }\, \Omega \ ;\ \  u \in W_0^{1,n}(\Omega).
        %  \end{alignedat}
    % \right.
%\end{equation}
%where $\mu\in (0,n)$, $\Om$ is a smooth bounded domain in $\mb R^n$, $n \geq 2$, $ 0< q < n-1 < 2n-1 < p+1 $, $m :\mb R^+ \to \mb R^+$ is real valued continuous functions such that $m(s)= as+b, a,b >0 $ and $  h \in L^{\gamma}(\Omega), \gamma = \frac{p+2+\beta}{1+q},  h^+ \not\equiv 0$ and $f(u)=u|u|^p exp(|u|^{\beta}),\ 0<\beta \leq \frac{n}{n-1}$
Using the Nehari manifold technique, we show existence and multiplicity of solutions with respect to the parameter $\la.$ Precisely, we show the following main results in the subcritical and critical case:
\begin{Theorem}\label{first}
Let $\beta \in \left( 1, \frac{n}{n-1}\right)$. Then there exists $\la_0$ such that  $(\mathcal{P}_{\la , M})$ admits at least two solutions for $\la \in (0,\la_0).$
\end{Theorem}
In the critical case, we show the following existence result.
\begin{Theorem}\label{second}
Let $\beta = \frac{n}{n-1}$, then there exists $\la_1 >0$ such that for $\la \in (0, \la_1),$ $\mathcal{J}_{\la, M}$ admits a solution.
\end{Theorem}
\section{Existence of a positive weak solution to $(KC)$}
In this section, we study problem $(KC)$ and for that we use the mountain pass Theorem and analyze accurately the compactness of Palais Smale sequences for $E$. First we show that the energy functional $E$ possesses the mountain pass geometry.
\begin{Lemma}\label{lemma3.1}
Assume the assumptions (m1), (m2) and (h1)-(h4). $E$ has the Mountain pass geometry around $0$.
\end{Lemma}
\begin{proof}
Let $u\in W^{1,n}_0(\Om)$ such that $\|u\|$ small enough. Let $0<\beta_0<\ell$. Then from Proposition \ref{HLS}, (h3) and \eqref{kc-1}, for any $\e>0$ we know that there exists a $C(\e)>0$ such that
\begin{equation}\label{kc-MP1}
\begin{split}
&\int_{\Om}\left(\int_\Om \frac{F(y,u)}{|x-y|^{\mu}}dy\right)F(x,u)~dx  \leq C(n,\mu)\|F(x,u)\|_{L^\frac{2n}{2n-\mu}(\Om)}^2\\
 &\leq C(n,\mu)2^{\frac{2n}{2n-\mu}}\left(\e \int_\Om |u|^{\frac{2n(\beta_0+1)}{2n-\mu}} +C(\e)\int_\Om |u|^{\frac{2pn}{2n-\mu}}\exp\left(\frac{2n(1+\e)}{2n-\mu}|u|^{\frac{n}{n-1}} \right) \right)^{\frac{2n-\mu}{n}}\\
 &\leq C_1 \left(\e \int_\Om |u|^{\frac{2n(\beta_0+1)}{2n-\mu}} + C_2(\e) \|u\|^{\frac{2pn}{2n-\mu}} \left( \int_\Om\exp\left(\frac{4n(1+\e)\|u\|^{\frac{n}{n-1}}}{2n-\mu}\left(\frac{|u|}{\|u\|}\right)^{\frac{n}{n-1}}\right)\right)^{\frac12} \right)^{\frac{2n-\mu}{n}}
 \end{split}
\end{equation}
where we used Sobolev and H\"{o}lder inequality.  So if we choose $\e>0$ small enough  and $u$ such that $\displaystyle\frac{4n(1+\e)\|u\|^{\frac{n}{n-1}}}{2n-\mu} \leq \alpha_n$ then using Theorem \ref{TM-ineq} in \eqref{kc-MP1} we get
\begin{align*}
\int_{\Om}\left(\int_\Om \frac{F(y,u)}{|x-y|^{\mu}}dy\right)F(x,u)~dx &\leq C_3 \left(\e \|u\|^{\frac{2n(\beta_0+1)}{2n-\mu}}  + C(\e) \|u\|^{\frac{2pn}{2n-\mu}}
 \right)^{\frac{2n-\mu}{n}}\\
 & \leq C_4 \left(\e \|u\|^{2(\beta_0+1)}  + C(\e) \|u\|^{2p}
 \right).
 \end{align*}
Hence from (m1) and above estimate, we deduce that for $\|u\|=\rho$ where $\rho<\left(\frac{\alpha_n(2n-\mu)}{4pn(1+\e)}\right)^{\frac{n-1}{n}}$
\begin{align*}
E(u) &\geq m_0\frac{\|u\|^n}{n}-   C_4 \left(\e \|u\|^{2(\beta_0+1)}  + C(\e) \|u\|^{2p}
 \right).
\end{align*}
Taking $\beta_0 >0$ such that $2(\beta_0+1) >n$ and $2p>n$, we can choose $\rho$ small enough so that $E(u) \geq \sigma$ for some $\sigma>0$ (depending on $\rho$) when $\|u\|=\rho$. Furthermore, under the assumption (m2), for some $a_1,\;a_2>0$ and $t_0>0$ we have $m(t) \leq a_1+a_2t^r$  and
\begin{equation*}
M(t)\leq \left\{
\begin{split}
&a_0+a_1t + \frac{a_2t^{r+1}}{r+1},\;r\neq -1\\
&a_0+a_1t + a_2\ln t,\;r= -1
\end{split}
\right.
\end{equation*}
when $t\geq \hat t$ and where
\begin{equation*}
a_0=\left\{
\begin{split}
&M(t_0)-a_1t_0-a_2\frac{t_0^{r+1}}{r+1},\;r\neq -1\\
&M(t_0) - a_1t_0 - a_2\ln t_0,\;r= -1.
\end{split}
\right.
\end{equation*}
Let $u_0 \in W^{1,n}_0(\Om)$ such that $u_0\geq  0$ and $\|u_0\|=1$. Then (h3) implies that there exists $K_1 \geq \max \{\frac{n}{2},\frac{n(r+1)}{2}\}$ such that $F(x,s) \geq C_1s^{K_1}-C_2$ for all $(x,s) \in \Om \times [0,\infty)$ and for some positive constants $C_1$ and $C_2$. Using this, we obtain
\begin{align*}
\int_\Om \left(\int_\Om \frac{F(y,tu_0)}{|x-y|^{\mu}}dy\right)F(x,tu_0)~dx &\geq \int_\Om \int_\Om \frac{(C_1(tu_0)^{K_1}(y)-C_2)(C_1(tu_0)^{K_1}(x)-C_2)}{|x-y|^{\mu}}~dxdy\\
  & = C_1^2 t^{2K_1} \int_\Om \int_\Om \frac{u_0^{K_1}(y)u_0^{K_1}(x)}{|x-y|^\mu}~dxdy \\
  & \quad -2C_1C_2t^{K_1}\int_\Om \int_\Om\frac{u_0^{K_1}(y)}{|x-y|^\mu}~dxdy + C_2^2 \int_\Om \int_\Om |x-y|^{-\mu}~dxdy.
\end{align*}
%\[0< \mc A(t) :=  \int_\Om \left(\int_\Om \frac{F\left(y,\frac{tu_0}{\|u_0\|}\right)}{|x-y|^{\mu}}dy\right) F\left(x,\frac{tu_0}{\|u_0\|}\right)~dx\geq \int_\Om \left(\int_\Om \frac{C_1}{|x-y|^{\mu}}dy\right) F\left(x,\frac{tu_0}{\|u_0\|}\right)~dx,\;\text{for}\;t>0.\]
%Using (h4), for every $t>0$ we get
%\begin{align*}
%\frac{\mc A^\prime(t)}{\mc A(t)} &= \frac{\int_\Om \left( \int_\Om \frac{F\left(y,\frac{tu_0}{\|u_0\|}\right)}{|x-y|^{\mu}}dy\right) f\left(x,\frac{tu_0}{\|u_0\|}\right)\frac{tu_0}{\|u_0\|}~dx}{ t\int_\Om \left(\int_\Om \frac{F\left(y,\frac{tu_0}{\|u_0\|}\right)}{|x-y|^{\mu}}dy\right) F\left(x,\frac{tu_0}{\|u_0\|}~dx\right)}\geq \frac{K_1}{t}.
%%& \geq \frac{\int_\Om \left(|x|^{-\mu}\ast F\left(\frac{tu_0}{\|u_0\|}\right)\right)\frac{K_1}{t} F\left(\frac{tu_0}{\|u_0\|}\right)}{\int_\Om \left(|x|^{-\mu}\ast F\left(\frac{tu_0}{\|u_0\|}\right)\right) F\left(\frac{tu_0}{\|u_0\|}\right)= \frac{K_1}{t}.
%\end{align*}
%So integrating this over $[1,t\|u_0\|]$ where $t>\frac{R_0}{\|u_0\|}$ and using (h3) we obtain
%\[\int_\Om \left(\int_\Om \frac{F(y,tu_0)}{|x-y|^{\mu}}dy\right)F(x,tu_0)~dx \geq t^{K_1}\|u_0\|^{K_1}\int_\Om \left(\int_\Om\frac{ F\left(y,\frac{u_0}{\|u_0\|}\right)}{|x-y|^{\mu}}dy\right)F\left(x,\frac{u_0}{\|u_0\|}\right)~dx. \]
Therefore from above we obtain
\begin{align*}
E(tu_0) &\leq \frac{M(\|t u_0\|^n)}{n}- \int_{\Om}\left( \int_\Om\frac{F(y, tu_0)}{|x-y|^{\mu}} dy\right)F(x, tu_0)~dx\\ &\leq C_3+ C_4t^n+ C_5t^{n(r+1)} - C_4t^{2K_1}+C_6t^{K_1}
\end{align*}
where $C_i's$ are positive constants for $i=4,5,6$. This implies that $E(tu_0) \to -\infty$ as $t \to \infty$. Thus there exists a $u_0\in W^{1,n}_0(\Om)$ with $\|u_0\|> \sigma$ such that $E(u_0)<0$.\hfill{\QED}
\end{proof}

\begin{Lemma}\label{kc-PS-bdd}
Every Palais Smale sequence is bounded in $W^{1,n}_0(\Om)$.
\end{Lemma}
\begin{proof}
Let $\{u_k\} \subset W^{1,n}_0(\Om)$ denotes a $(PS)_c$ sequence of $E$ that is
\begin{equation*}
E(u_k) \to c \; \text{and}\; E^\prime(u_k) \to 0\;\text{as}\; k \to \infty
\end{equation*}
for some $c \in \mathbb{R}.$ This implies
\begin{equation}\label{kc-PS-bdd1}
\begin{split}
&\frac{M(\|u_k\|^n)}{n} - \frac12 \int_\Om \left(\int_\Om \frac{F(y,u_k)}{|x-y|^{\mu}}dy \right)F(x,u_k)~dx \to c \; \text{as}\; k \to \infty,\\
&\left| m(\|u_k\|^n)\int_\Om |\nabla u_k|^{n-2}\nabla u_k \nabla\phi -\int_\Om \left(\int_\Om \frac{F(y,u_k)}{|x-y|^{\mu}}dy \right)f(x,u_k)\phi ~dx \right|\leq \e_k\|\phi\|
\end{split}
\end{equation}
where $\e_k \to 0$ as $k\to \infty$. In particular, taking $\phi=u_k$ we get
\begin{equation}\label{kc-PS-bdd2}
\left| m(\|u_k\|^n)\int_\Om |\nabla u_k|^{n}-\int_\Om \left(\int_\Om \frac{F(y,u_k)}{|x-y|^{\mu}}dy \right)f(u_k)u_k ~dx \right|\leq \e_k\|u_k\|.
\end{equation}
From the assumption (h3), there exists $\alpha>n$ such  that $\alpha F(x,t)\leq tf(x,t)$ for any $t>0$ and $ x\in \Om$ which yields
\begin{equation}\label{kc-PS-bdd3}
\alpha \int_\Om \left(\int_\Om \frac{F(y,u_k)}{|x-y|^{\mu}}dy \right)F(u_k)~dx \leq \int_\Om \left(\int_\Om \frac{F(y,u_k)}{|x-y|^{\mu}}dy \right)f(u_k)u_k ~dx.
\end{equation}
%We notice that because of \eqref{kc-1} and Remark \ref{rem2.2} we get
%\[\int_\Om \int_\Om \frac{F(y,u_k)}{|x-y|^{\mu}}~dxdy \leq \left(\int_\Om |F(y,u_k)|^{q^\prime}dy\right)^{\frac{1}{q^\prime}} \int_\Om \left(\int_\Om \frac{1}{|x-y|^{\mu q}}dy\right)^{\frac{1}{q}}~dx \leq C_0(\|u_k\|^{\beta_0+1} + \|u_k\|^p )  \]
%where $q\in [1,2n/\mu)$.}
Using  \eqref{kc-PS-bdd1}, \eqref{kc-PS-bdd2} along with above inequality and $(m3)^\prime$, we get
\begin{equation}\label{kc-PS-bdd4}
\begin{split}
& E(u_k)- \frac{1}{2\alpha}\langle E^\prime(u_k),u_k\rangle
=\frac{M(\|u_k\|^n)}{n}- \frac{m(\|u_k\|^n)\|u_k\|^n}{2 \alpha}  \\
& \quad \quad-\frac12 \left( \int_\Om \left(\int_\Om \frac{F(y,u_k)}{|x-y|^{\mu}}dy \right)F(x,u_k)~dx- \frac{1}{\alpha}\int_\Om \left(\int_\Om \frac{F(y,u_k)}{|x-y|^{\mu}}dy \right)f(x,u_k)u_k ~dx\right)\\
&{\geq \frac{M(\|u_k\|^n)}{n}- \frac{m(\|u_k\|^n)\|u_k\|^n}{2 \alpha}}  {\geq \left( \frac{1}{2n}- \frac{1}{2 \alpha}\right) m(\|u_k\|^n)\|u_k\|^n  \geq  \left( \frac{1}{2n}- \frac{1}{2 \alpha}\right) m_0 \|u_k\|^n}.
\end{split}
\end{equation}
Also from \eqref{kc-PS-bdd1} and \eqref{kc-PS-bdd2} it follows that
\begin{equation}\label{kc-PS-bdd5}
E(u_k)- \frac{1}{2\alpha}\langle E^\prime(u_k),u_k\rangle \leq C \left( 1+ \e_k \frac{\|u_k\|}{2\alpha}\right)
\end{equation}
for some constant $C>0$. Therefore from \eqref{kc-PS-bdd4} and \eqref{kc-PS-bdd5} we get that
\[ \left( \frac{1}{2n}- \frac{1}{2 \alpha}\right) m_0 \|u_k\|^n  \leq C \left( 1+ \e_k \frac{\|u_k\|}{2\alpha}\right).\]
This implies that $\{u_k\}$ must be bounded in $W^{1,n}_0(\Om)$. \hfill{\QED}
\end{proof}

Let $\Gamma = \{\gamma \in C([0,1], W^{1,n}_0(\Om)):\; \gamma(0)=0, \;E(\gamma(1))<0\}$ and define the Mountain Pass critical level as\
\begin{equation}\label{*}
l^* = \inf_{\gamma \in \Gamma}\max_{t\in[0,1]} E(\gamma(t)).
\end{equation}
Then we have the following result:

\begin{Lemma}\label{PS-level}
If \eqref{h-growth} hold,  then $$0< l^* < \displaystyle \frac{1}{n} M \left(\left( \frac{2n-\mu}{2n}\alpha_n\right)^{n-1}\right).$$
\end{Lemma}
\begin{proof}
Since for $u\not\equiv 0$, $E(tu) \to -\infty$ as $t\to \infty$ (as we proved in Lemma~\ref{lemma3.1}),   $l^* \leq \max_{t\in[0,1]} E(tu)$ for $u \in W^{1,n}_0(\Omega)\backslash\{0\}$. So it is enough to show that there exists a $w \in W^{1,n}_0(\Om)$ such that $\|w\|=1$ and
\[\max_{t\in[0,\infty)} E(tw) < \frac{1}{n} M\left(\left( \frac{2n-\mu}{2n}\alpha_n\right)^{n-1}\right).\]
To prove this, we consider the sequence of Moser functions $\{w_k\}$ defined as
\begin{equation*}
w_k(x)=\frac{1}{\omega_{n-1}^{\frac{1}{n}}}\left\{
\begin{split}
& (\log k)^{\frac{n-1}{n}},\; 0\leq |x|\leq \frac{\rho}{k},\\
& \frac{\log \left(\frac{\rho}{|x|}\right)}{(\log k)^{\frac{1}{n}}}, \; \frac{\rho}{k}\leq |x|\leq \rho,\\
& 0,\; |x|\geq \rho
\end{split}
\right.
\end{equation*}
so that supp$(w_k) \subset B_\rho(0)$. It is easy to verify that $\|w_k\| =1$ for all $k$. So we claim that there exists a $k \in \mb N$ such that
\[\max_{t\in[0,\infty)} E(tw_k) < \frac{1}{n} M\left(\left( \frac{2n-\mu}{2n}\alpha_n\right)^{n-1}\right).\]
Suppose this is not true then for all $k \in \mb N$ there exists a $t_k>0$ such that
\begin{equation}\label{kc-PScond0}
\begin{split}
&\max_{t\in[0,\infty)} E(tw_k) = E(t_kw_k) \geq \frac{1}{n} M\left(\left( \frac{2n-\mu}{2n}\alpha_n\right)^{n-1}\right)\\
& \text{and}\;  \frac{d}{dt}(E(tw_k))|_{t=t_k}=0.
\end{split}
\end{equation}
From the proof of Lemma \ref{kc-PS-bdd}, $E(t w_k)\to -\infty$ as $t\to \infty$ uniformly in $k$. Then we infer that $\{t_k\}$ must be a bounded sequence in $\mb R$. From \eqref{kc-PScond0} and definition of $E(t_kw_k)$ we obtain
\begin{equation}\label{kc-PScond1}
\frac{1}{n} M\left(\left( \frac{2n-\mu}{2n}\alpha_n\right)^{n-1}\right) < \frac{M(t_k^n)}{n}.
\end{equation}
Since $M$ is monotone increasing, from \eqref{kc-PScond1} we get that
\begin{equation}\label{kc-PScond2}
t_k^n \geq \left( \frac{2n-\mu}{2n}\alpha_n\right)^{n-1}.
\end{equation}
From \eqref{kc-PScond2}, we get
\begin{equation}\label{kc-PScond3}
\frac{t_k}{\omega_{n-1}^{\frac{1}{n}}}(\log k)^{\frac{n-1}{n}} \to \infty \;\text{as}\; k \to \infty.
\end{equation}
Furthermore from \eqref{kc-PScond0},  we have
\begin{equation}\label{kc-PS-cond3}
\begin{split}
m(t_k^n)t_k^n &= \int_\Om \left(\int_\Om \frac{F(y,t_kw_k)}{|x-y|^{\mu}}dy\right)f(x,t_kw_k)t_kw_k ~dx\\
& \geq \int_{B_{\rho/k}}f(x,t_kw_k)t_kw_k \int_{B_{\rho/k}}\frac{F(y,t_kw_k)}{|x-y|^\mu}~dy~dx.
%& \geq (\beta-\e)\exp\left( a \left( \frac{t_k(\log k)^{\frac{n-1}{n}}}{\omega_{n-1}^{\frac{1}{n}}}\right)^{\frac{n}{n-1}}\right)\int_{B_{\rho/k}}\int_{B_{\rho/k}} \frac{~dxdy}{|x-y|^\mu}
\end{split}
\end{equation}
 In addition, as in equation $(2.11)$ p. 1943 in \cite{yang-JDE}, it is easy to get that
\[\int_{B_{\rho/k}}\int_{B_{\rho/k}} \frac{~dxdy}{|x-y|^\mu} \geq C_{\mu, n} \left(\frac{\rho}{k}\right)^{2n-\mu}\]
where $C_{\mu, n}$ is a positive constant depending on $\mu$ and $n$. From \eqref{h-growth}, we know that for each $d>0$ there exists a $s_d$ such that
\[sf(x,s)F(x,s) \geq d \exp\left( 2|s|^{\frac{n}{n-1}}\right),\; \text{whenever}\; s \geq s_d.\]
Since \eqref{kc-PScond3} holds, we can choose a $r_d \in \mb N$ such that
\[\frac{t_k}{\omega_{n-1}^{\frac{1}{n}}}(\log k)^{\frac{n-1}{n}} \geq s_d,\; \text{for all}\; k\geq r_d.\]
 Using these estimates in \eqref{kc-PS-cond3} and from \eqref{kc-PScond2}, for $d$ large enough
  we get that
\begin{equation*}
m(t_k^n)t_k^n \geq d \exp \left( (\log k)\left(\frac{2 t_k^{\frac{n}{n-1}}}{\omega_{n-1}^{\frac{1}{n-1}}}\right)\right)C_{\mu,n}\left(\frac{\rho}{k}\right)^{2n-\mu} \geq d  C_{\mu, n}\rho^{2n-\mu}.
\end{equation*}
%which implies
%\begin{equation}
%\frac{m(t_k^n)t_k^n}{k^{{a t_k^{\frac{n}{n-1}}}{\omega_{n-1}^{-\frac{1}{n-1}}}-(4-\mu)}}\geq (\beta -\e)C_\mu\rho^{4-\mu}.
%\end{equation}
Taking $d$ large enough and since $t_k^n$ is bounded, we arrive at a contradiction. This establishes our claim and we conclude the proof of the result. \hfill{\QED}
\end{proof}

\begin{Lemma}\label{wk-sol}
If $\{u_k\}$ denotes a Palais Smale sequence then up to a subsequence, there exists $u \in W^{1,n}_0(\Om)$ such that
%\begin{equation}\label{wk-sol1}
%(|x|^{-\mu}\ast F(u_k))f(u_k) \to (|x|^{-\mu}\ast F(u))f(u) \; \text{in}\; L^1(\Om)
%\end{equation}
\begin{equation}\label{wk-sol2}
 |\nabla u_k|^{n-2}\nabla u_k \rightharpoonup |\nabla u|^{n-2}\nabla u\; \text{weakly in}\; (L^{\frac{n}{n-1}}(\Om))^n.
\end{equation}
\end{Lemma}
\begin{proof}
From Lemma \ref{kc-PS-bdd}, we know that the sequence $\{u_k\}$ must be bounded in $W^{1,n}_0(\Om)$. Consequently, up to a subsequence, there exists $u\in W^{1,n}_0(\Omega)$ such that $u_k \rightharpoonup u$ weakly in $W_0^{1,n}(\Om)$ and strongly in $L^q(\Om)$ for any $q \in [1,\infty)$ as $k \to \infty$. Also still up to a subsequence we can assume $u_k(x) \to u(x)$ pointwise a.e. for $ x \in \Om$. Therefore the sequence $\{|\nabla u_k|^{n-2}\nabla u_k\}$ must be bounded in $(L^{\frac{n}{n-1}}(\Om))^n$ whereas $|\nabla u|^n$ is bounded in $L^1(\Om)$. So we use that there exists a non-negative radon measure $\nu$ such that up to a subsequence
\[|\nabla u_k|^n \to \nu \; \text{in}\; (C(\overline{\Om}))^*\; \text{as}\; k \to \infty.\]
Moreover there exists $v \in (L^{\frac{n}{n-1}}(\Om))^n$ such that,
\[|\nabla u_k|^{n-2}\nabla u_k \to v \; \text{weakly in}\;  (L^{\frac{n}{n-1}}(\Om))^n \; \text{as} \; k \to \infty.\]
\textbf{Claim :} $v = |\nabla u|^{n-2}\nabla u $.\\
To prove this, we set $\sigma>0$ and $X_\sigma = \{x \in \overline \Om:\; \nu(B_r(x)\cap \overline \Om)\geq \sigma, \;\text{for all}\; r>0\}$. Then $X_\sigma$ must be a finite set.  Because if not, then there exists a sequence of distinct points $\{x_k\}$ in $X_\sigma$ such that for all $r>0$, $\nu(B_r(x_k)\cap \overline \Om)\geq \sigma$ for all $k$. This implies that $\nu(\{x_k\}) \geq \sigma$ for all $k$, hence $\nu(X_\sigma)= +\infty$. But this is a contradiction to
\[\nu(X_\sigma) = \lim_{k \to \infty} \int_{X_\sigma} |\nabla u_k|^n ~dx \leq C.\]
So let $X_\sigma= \{x_1,x_2,\ldots, x_m\}$. Next, we claim that if we take  $\sigma>0$ such that $\sigma^{\frac{1}{n-1}} <  \frac{2n-\mu}{2n} \alpha_n$, the  for any $K$ compact subset of $\oline \Om \setminus X_\sigma$ we have
\begin{equation}\label{wk-sol6}
\lim_{k \to \infty} \int_K\left( \int_\Om \frac{F(y,u_k)}{|x-y|^{\mu}}dy\right) f(x,u_k)u_k~dx =  \int_K \left( \int_\Om \frac{F(y,u)}{|x-y|^{\mu}}dy\right) f(x,u)u~dx. \end{equation}
To show this, let $x_0 \in K$ and $r_0>0$ be such that $\nu(B_{r_0}(x_0) \cap \oline \Om) < \sigma$ that is $x_0 \notin X_\sigma$. Also we consider a $\psi \in C^\infty(\Om)$ satisfying $0\leq \psi(x)\leq 1$ for $x \in \Om$, $\psi \equiv 1$ in $B_{\frac{r_0}{2}}(x_0)\cap \oline \Om$ and $\psi \equiv 0$ in $\oline \Om \setminus (B_{r_0}(x_0)\cap \oline \Om)$. Then
\[\lim_{k \to \infty} \int_{B_{\frac{r_0}{2}}(x_0) \cap \oline \Om}|\nabla u_k|^n \leq \lim_{k \to \infty} \int_{B_{r_0}(x_0) \cap \oline \Om}|\nabla u_k|^n\psi \leq \nu(B_{r_0}(x_0) \cap \oline \Om) < \sigma. \]
Therefore for large enough $k \in \mb N$ and $\e>0$ small enough, it must be
\begin{equation}\label{wk-sol3}
 \int_{B_{\frac{r_0}{2}}(x_0) \cap \oline \Om}|\nabla u_k|^n \leq \sigma(1-\e).
 \end{equation}
Now we estimate the following using \eqref{wk-sol3} and Theorem \ref{TM-ineq}
\begin{equation}\label{wk-sol4}
\begin{split}
%&\int_{B_{\frac{r_0}{2}}(x_0) \cap \oline \Om} \left(\int_\Om \frac{F(x,u_k)}{|x-y|^\mu}dy\right)|f(x,u_k)|^q~dx
& \int_{B_{\frac{r_0}{2}}(x_0) \cap \oline \Om}|f(x,u_k)|^q~dx=  \int_{B_{\frac{r_0}{2}}(x_0) \cap \oline \Om}|h(x,u_k)|^q\exp\left(q|u_k|^{\frac{n}{n-1}}\right)~dx \\
&  \leq C_\delta \int_{B_{\frac{r_0}{2}}(x_0) \cap \oline \Om}\exp\left((1+\epsilon)q|u_k|^{\frac{n}{n-1}}\right)~dx\\
& \leq C_\delta \int_{B_{\frac{r_0}{2}}(x_0) \cap \oline \Om}\exp\left((1+\epsilon)q\sigma^{\frac{1}{n-1}}(1-\e)^{\frac{1}{n-1}}\left(\frac{|u_k|^n}{\int_{B_{\frac{r_0}{2}}(x_0) \cap \oline \Om}|\nabla u_k|^n}\right)^{\frac{1}{n-1}}\right)~dx\leq C_0
\end{split}
\end{equation}
for some constant $C_0>0$ while choosing $q>1$ such that $(1+\e)q \sigma^{\frac{1}{n-1}} \leq \alpha_n$. Consider
\begin{equation*}
\begin{split}
 &\int_{B_{\frac{r_0}{2}}(x_0) \cap \oline \Om} \left|\left( \int_\Om \frac{F(y,u_k)}{|x-y|^{\mu}}dy\right) f(x,u_k)u_k- \left( \int_\Om \frac{F(y,u)}{|x-y|^{\mu}}dy\right) f(x,u)u \right|~dx\\
 & \leq \int_{B_{\frac{r_0}{2}}(x_0) \cap \oline \Om} \left|\left( \int_\Om \frac{F(y,u)}{|x-y|^{\mu}}dy\right) (f(x,u_k)u_k-f(x,u)u)\right|~dx\\
 & \quad + \int_{B_{\frac{r_0}{2}}(x_0) \cap \oline \Om} \left|\left( \int_\Om \frac{F(y,u_k)-F(y,u)}{|x-y|^{\mu}}dy\right) f(x,u_k)u_k\right|~dx\\
 &:= I_1 + I_2 \;\text{(say)}.
\end{split}
\end{equation*}
From \eqref{KC-new1}, we know that $F(u) \in L^r(\Om)$ for any $ r \in [1,\infty)$. Since $\mu \in (0,n)$, $y\to |x-y|^{-\mu} \in L^{r_0}(\Om)$ for all $r_0 \in (1, \frac{n}{\mu})$ uniformly in $x\in \Omega$ (since $\Om$ is bounded). So using H\"older's inequality we get that
\begin{equation}\label{wk-sol7}
\int_\Om \frac{F(y,u)}{|x-y|^{\mu}}dy \in L^\infty(\Om).
\end{equation}
From the asymptotic growth of $f(x,t)$, it is easy to get that
\begin{equation}\label{wk-sol8}
\lim_{t \to \infty} \frac{f(x,t)t}{(f(x,t))^r} = 0\; \text{uniformly in }x \in \Omega,\; \text{for all}\; r>1.
\end{equation}
Using \eqref{wk-sol7} we get
\[I_1 \leq  C\int_{B_{\frac{r_0}{2}}(x_0) \cap \oline \Om}  |f(x,u_k)u_k-f(x,u)u|~dx\]
where $C>0$ is a constant. Because of \eqref{wk-sol8} and \eqref{wk-sol4}, the family $\{f(x,u_k)u_k\}$ is equi-integrable over ${B_{\frac{r_0}{2}}(x_0) \cap \oline \Om} $. Also continuity of $f(x,t)$ gives that $f(x,u_k)u_k\to f(x,u)u$ pointwise a.e. in $\Om$ as $k \to \infty$ and  thus using Vitali's convergence theorem, it follows that $I_1 \to 0$ as $k\to \infty$. Next we show $I_2 \to 0$ as $k \to \infty$.\\
First by using the semigroup property of the Riesz Potential we get that for some constant $C>0$ independent of $k$
\begin{align*}
\begin{split}
&\int_{\Omega} \left(\int_{\Omega}\frac{F(y,u_k)-F(y,u)}{|x-y|^{\mu}} dy \right) \chi_{B_{\frac{r_0}{2}} \cap \overline{\Omega}}(x) f(x,u_k) u_k ~dx \\
&\leq \left(\int_{\Omega} \left( \int_{\Omega}\frac{|F(y,u_k)- F(y,u)| dy }{|x-y|^{\mu}}\right) |F(x,u_k)- F(x,u)| ~dx \right)^{\frac12}\\
&\quad \times \left(\int_{\Omega} \left(\int_{\Omega} \chi_{B_{\frac{r_0}{2}} \cap \overline{\Omega}}(y) \frac{f(y,u_k) u_k}{|x-y|^{\mu}} dy \right) \chi_{B_{\frac{r_0}{2}} \cap \overline{\Omega}}(x) f(x,u_k) u_k ~dx \right)^{\frac12}.
\end{split}
 \end{align*}
From \eqref{wk-sol4} and since $\sigma^{\frac{1}{n-1}} < \frac{2n-\mu}{2n} \alpha_n$ we obtain
\begin{align*}
\left(\int_{\Omega} \left(\int_{\Omega} \chi_{B_{\frac{r_0}{2}} \cap \overline{\Omega}}(y) f(y,u_k) u_k dy \right) \chi_{B_{\frac{r_0}{2}} \cap \overline{\Omega}}(x) f(x,u_k) u_k ~dx \right)^{\frac12} \leq \|\chi_{B_{\frac{r_0}{2}} \cap \overline{\Omega}}f(x,u_k) u_k\|_{L^{\frac{2n}{2n-\mu}}(\Omega)} \leq C.
\end{align*}
Now we claim that
\begin{align}\label{3.23}
\lim_{k \to \infty} \int_{\Omega} \left(\int_{\Omega}\frac{|F(y,u_k)-F(y,u)|}{|x-y|^{\mu}} dy \right) |F(x,u_k)-F(x,u)|  ~dx =0.
\end{align}
From \eqref{kc-PS-bdd1}, \eqref{kc-PS-bdd2} and \eqref{kc-PS-bdd3} we get that there exists a constant $C>0$ such that
\begin{equation}\label{wk-sol10}
\begin{split}
\int_\Om \left(\int_\Om\frac{F(y,u_k)}{|x-y|^\mu}dy\right)F(x,u_k)~dx &\leq C,\\
\int_\Om \left(\int_\Om\frac{F(y,u_k)}{|x-y|^\mu}dy\right)f(x,u_k)u_k~dx & \leq C.
\end{split}
\end{equation}
We argue as along equation $(2.20)$ in Lemma $2.4$ in \cite{yang-JDE}. Now using \eqref{wk-sol10}, (h4) and the semigroup property of the Riesz Potential we obtain,
\begin{equation}\label{3.25}
\int_{\Omega} \int_{|u| \geq M} \frac{F(y,u)}{|x-y|^{\mu}} F(x,u) dy ~dx=o(M), \;\int_{\Omega} \int_{|u_k| \geq M} \frac{F(y,u_k)}{|x-y|^{\mu}} F(x,u_k) dy ~dx = o(M),
\end{equation}
\begin{equation}\label{3.26}
\int_{\Omega} \int_{|u| \geq M} \frac{F(y,u_k)}{|x-y|^{\mu}} F(x,u) dy ~dx = o(M),
\end{equation}
and
\begin{equation}\label{3.27}
\int_{\Omega} \int_{|u_k| \geq M} \frac{F(y,u_k)}{|x-y|^{\mu}} F(x,u) dy ~dx = o(M) \ \text{as}\ M \to \infty.
\end{equation}
So,
\begin{equation*}
\begin{split}
& \int_{\Omega} \left(\int_{\Omega}\frac{|F(y,u_k)-F(y,u)|}{|x-y|^{\mu}} dy \right) |F(x,u_k)-F(x,u)|  ~dx\leq
2\int_{\Omega} \left(\int_{\Omega}\frac{\chi_{u_k\geq M}(y)F(y,u_k)}{|x-y|^{\mu}} dy \right) F(x,u_k) ~dx \\
&+4 \int_{\Omega} \left(\int_{\Omega}\frac{F(y,u_k)\chi_{u \geq M}(x)F(x,u)}{|x-y|^{\mu}} dy \right) ~dx+4 \int_{\Omega} \left(\int_{\Omega}\frac{\chi_{u_k\geq M}(y)F(y,u_k)F(x,u)}{|x-y|^{\mu}} dy \right) ~dx\\
&+2\int_{\Omega} \left(\int_{\Omega}\frac{\chi_{u\geq M}(y)F(y,u)}{|x-y|^{\mu}} dy \right) F(x,u) ~dx \\
&+\int_\Omega\left(\int_\Omega\frac{|F(y,u_k)\chi_{u_k\leq M}-F(y,u)\chi_{u\leq M}|}{|x-y|^\mu}dy\right)|F(x,u_k)\chi_{u_k\leq M}-F(x,u)\chi_{u\leq M}|~dx.
\end{split}
\end{equation*}
Then from Lebesgue dominated convergence theorem the above integrand tends to $0$ as $k \to \infty.$ Hence using \eqref{3.25}, \eqref{3.26} and \eqref{3.27}, it is easy to conclude \eqref{3.23} and $I_2 \to 0$ as $k \to \infty$.
%Lastly we estimate the following while choosing $K > \max\{T, \}$, using (h5)
%\begin{equation}
%\begin{split}
%&\int_{\{|u|\leq M\}}\left(\int_{\{|u_k|\geq K\}}\frac{F(y,u_k)}{|x-y|^\mu}dy\right)F(x,u)~dx \\
%& = \int_{\{|u|\leq M\}}\left(\int_{\{|u_k|\leq K\}\cap \{F(u_k)\leq 1\}}\frac{F(y,u_k)}{|x-y|^\mu}dy\right)F(x,u)~dx + \int_{\{|u|\leq M\}}\left(\int_{\{|u_k|\leq K\}\cap\{F(u_k)>1\}}\frac{F(y,u_k)}{|x-y|^\mu}dy\right)F(x,u)~dx\\
%&\leq \int_{\{|u|\leq M\}}\left(\int_{\{|u_k|\leq K\}}\frac{1}{|x-y|^\mu}dy\right)F(x,u)~dx + \frac{T_0}{K^{\gamma_0+1}}\int_{\{|u|\leq M\}}\left(\int_{\{|u_k|\leq K\}\cap \{F(u_k)\leq 1\}}\frac{u_k f(y,u_k)}{|x-y|^\mu}dy\right)F(x,u)~dx
%\end{split}
%\end{equation}
This implies that
\[\lim_{k \to \infty}\int_{B_{\frac{r_0}{2}}(x_0) \cap \oline \Om} \left|\left( \int_\Om \frac{F(y,u_k)}{|x-y|^{\mu}}dy\right) f(x,u_k)u_k- \left( \int_\Om \frac{F(y,u)}{|x-y|^{\mu}}dy\right) f(x,u)u \right|~dx=0.\]
Now to conclude \eqref{wk-sol6}, we repeat this procedure over a finite covering of balls using the fact that $K$ is compact. Lastly, the proof of \eqref{wk-sol2} can be achieved by classical arguments as in the proof of Lemma $4$ in \cite{M-do1}. \hfill{\QED}
\end{proof}

\begin{Lemma}\label{kc-ws}
Let $\{u_k\}\subset W^{1,n}_0(\Om)$ be a Palais Smale sequence for $E$ at level $l^*$ then under the assumption (h4), there exists a $u_0\in W^{1,n}_0(\Om)$ such that as $k \to \infty$ (up to asubsequence)
\[\int_\Om \left(\int_\Om \frac{F(y,u_k)}{|x-y|^{\mu}}dy\right)f(x,u_k)\phi ~dx\to \int_\Om \left(\int_\Om \frac{F(y,u_0)}{|x-y|^{\mu}}dy\right)f(x,u_0)\phi~dx,\;\text{for all}\; \phi\in C_c^\infty(\Om). \]
\end{Lemma}
\begin{proof}
If $\{u_k\}$ is a Palais Smale sequence at $l^*$ for $E$ then it must satisfy \eqref{kc-PS-bdd1} and \eqref{kc-PS-bdd2}. We remark that $E(u^+)\leq E(u)$ for each $u \in W^{1,n}_0(\Om)$, then we can assume $u_k \geq 0$ for each $k \in \mb N$. From Lemma \ref{kc-PS-bdd} we know that $\{u_k\}$ must be bounded in $W^{1,n}_0(\Om)$ so there exists a $C_0>0$ such that $\|u_k\|\leq C_0$. Also there exists a $u_0 \in W^{1,n}_0(\Om)$ such that up to a subsequence $u_k \rightharpoonup u_0$ in $W^{1,n}_0(\Om)$, strongly in $L^{q}(\Om)$ for all $q \in [1,\infty)$ and pointwise a.e. in $\Om$ as $k \to \infty$.  Let $\Om' \subset\subset \Om$ and $\varphi \in C_c^\infty(\Om)$ such that $0\leq \varphi \leq 1$ and $\varphi \equiv 1$ in $\Om' $. With easy computations, we get that
\begin{equation*}
\begin{split}
\left\| \frac{\varphi}{1+u_k}\right\|^n &= \int_\Om \left|\frac{\nabla \varphi}{1+u_k}- \varphi \frac{\nabla u_k}{(1+u_k)^2} \right|^n~dx\\
 &\leq 2^{n-1}(\|\varphi\|^n+ \|u_k\|^n).
\end{split}
\end{equation*}
This implies that $\frac{\varphi}{1+u_k} \in W^{1,n}_0(\Om)$. So using $\frac{\varphi}{1+u_k}$ as a test function \eqref{kc-PS-bdd1}, we get the following estimate
\begin{equation*}
\begin{split}
&\int_{\Om^{'}}\left( \int_\Om \frac{F(y,u_k)}{|x-y|^\mu}dy\right)\frac{f(x,u_k)}{1+u_k}~dx \leq \int_\Om \left( \int_\Om \frac{F(y,u_k)}{|x-y|^\mu}dy\right)\frac{f(x,u_k)\varphi}{1+u_k}~dx\\
&\leq  \e_k \left\|\frac{\varphi}{1+u_k}\right\| + \int_\Om m(\|u_k\|^n) |\nabla u_k|^{n-2}\nabla u_k \nabla \left( \frac{\varphi}{1+u_k}\right)~dx\\
& \leq \e_k 2^{\frac{n-1}{n}}(\|\varphi\|+ \|u_k\|) + m(\|u_k\|^n)\int_\Om |\nabla u_k|^{n-2}\nabla u_k \left(\frac{\nabla \varphi}{1+u_k}-\varphi\frac{\nabla u_k}{(1+u_k)^2}\right)~dx\\
& \leq \e_k 2^{\frac{n-1}{n}}(\|\varphi\|+ \|u_k\|) + m(\|u_k\|^n)\int_\Om |\nabla u_k|^{n-1} \left( |\nabla \varphi|+ |\nabla u_k|\right)~dx\\
& \leq \e_k 2^{\frac{n-1}{n}}(\|\varphi\|+ \|u_k\|)+ m(\|u_k\|^n)[\|\varphi\|\|u_k\|^{n-1}+ \|u_k\|^n].
\end{split}
\end{equation*}
But using $\|u_k\|\leq C_0$ for all $k$ and (m2), we infer that there must exists a $C_1>0$ such that
\begin{equation}\label{kc-ws-new1}
\int_{\Om^{'}} \left(\int_\Om  \frac{F(y,u_k)}{|x-y|^\mu}dy\right)\frac{f(x,u_k)}{1+u_k}~dx  \leq C_1.
\end{equation}
Also for the same reason, \eqref{kc-PS-bdd2} gives that
\begin{equation}\label{kc-ws-new2}
\int_{\Om^{'}} \left( \int_\Om \frac{F(y,u_k)}{|x-y|^\mu}dy\right){f(x,u_k)}{u_k}~dx  \leq C_2
\end{equation}
for some $C_2>0$. Gathering \eqref{kc-ws-new1} and \eqref{kc-ws-new2} we obtain
\begin{equation*}
\begin{split}
&\int_{\Om^{'}}\left( \int_\Om \frac{F(y,u_k)}{|x-y|^\mu}dy\right){f(x,u_k)}~dx \\
& \leq  2\int_{\Om^{'}\cap \{u_k <1\}} \left( \int_\Om\frac{F(y,u_k)}{|x-y|^\mu}dy\right)\frac{f(x,u_k)}{1+u_k}~dx + \int_{\Om^{'}\cap \{u_k \geq 1\}} \left(\int_\Om \frac{F(y,u_k)}{|x-y|^\mu}dy\right)u_k{f(x,u_k)}~dx\\
& \leq 2\int_{\Om^{'}} \left( \int_\Om\frac{F(y,u_k)}{|x-y|^\mu}dy\right)\frac{f(x,u_k)}{1+u_k}~dx + \int_{\Om^{'}}\left( \int_\Om \frac{F(y,u_k)}{|x-y|^\mu}dy\right)u_k{f(x,u_k)}~dx\\
& \leq 2C_1+C_2  :=C_3.
\end{split}
\end{equation*}
Thus the sequence $\{w_k\}:=\left\{\left( \int_\Om\frac{F(y,u_k)}{|x-y|^\mu}dy\right){f(x,u_k)}\right\}$ is bounded in $L^1_{\text{loc}}(\Om)$ which implies that up to a subsequence, $w_k \rightharpoonup w$ in the ${weak}^*$-topology as $k \to \infty$, where $w$ denotes a Radon measure. So for any $\phi \in C_c^\infty(\Om)$ we get
\[\lim_{k \to \infty}\int_\Om\int_\Om \left( \frac{F(y,u_k)}{|x-y|^\mu}dy\right){f(x,u_k)}\phi ~dx = \int_\Om \phi ~dw,\; \forall \phi \in C_c^\infty(\Om). \]
Since $u_k$ satisfies \eqref{kc-PS-bdd1}, we get that
\[\int_E \phi dw= \lim_{k \to \infty} m(\|u_k\|)\int_E |\nabla u_k|^{n-2}\nabla u_k \nabla \phi ~dx, \;\;\forall E\subset \Om.  \]
Together with Lemma~\ref{wk-sol}, this implies that $w$ is absolutely continuous with respect to the Lebesgue measure. Thus, Radon-Nikodym theorem asserts that there exists a function $g \in L^1_{\text{loc}}(\Om)$ such that for any $\phi\in C^\infty_c(\Omega)$, $\int_\Om \phi~ dw= \int_\Om \phi g~dx$.
Therefore for any $\phi\in C^\infty_c(\Omega)$ we get
\[\lim_{k \to \infty}\int_\Om\left( \int_\Om \frac{F(y,u_k)}{|x-y|^\mu}dy\right){f(x,u_k)}\phi~ ~dx = \int_\Om \phi g~dx = \int_\Om  \left( \int_\Om \frac{F(y,u_0)}{|x-y|^\mu}dy\right){f(x,u_0)}\phi~ ~dx \]
which completes the proof.\hfill{\QED}
\end{proof}

In the next Lemma, we show that weak limit of any $(PS)_c$ sequence is a weak solution of $(KC)$.

\begin{Lemma}\label{PS-ws}
Let $\{u_k\}\subset W^{1,n}_0(\Om)$ be a Palais Smale sequence of $E$. Then there exists a $u \in W^{1,n}_0(\Om)$ such that, up to a subsequence, $u_k \rightharpoonup u$ weakly in $W^{1,n}_0(\Om)$ and
\begin{equation}\label{PS-wk0}
\left(\int_\Om \frac{F(y,u_k)}{|x-y|^{\mu}}dy\right)F(x,u_k)  \to \left(\int_\Om \frac{F(y,u)}{|x-y|^{\mu}}dy\right)F(x,u) \; \text{in}\; L^1(\Om)
\end{equation}
as $k \to \infty$. Moreover, $u$ forms a weak solution of $(KC)$.
\end{Lemma}
\begin{proof}
Let $\{u_k\}\subset W^{1,n}_0(\Om)$ be a Palais Smale sequence of $E$ at level $c$.
 From Lemma \ref{kc-PS-bdd} we know that $\{u_k\}$  must be bounded in $W^{1,n}_0(\Om)$. Thus there exists a $u\in W^{1,n}_0(\Om)$ such that $u_k \rightharpoonup u$ weakly in $W^{1,n}_0(\Om)$, $u_k \to u$ pointwise a.e. in $\mb R^n$ and $u_k \to u$ strongly in $L^q(\Om)$, $q \in [1,\infty)$ as $k \to \infty$. Also from \eqref{kc-PS-bdd1}, \eqref{kc-PS-bdd2} and \eqref{kc-PS-bdd3} we get that there exists a constant $C>0$ such that \eqref{wk-sol10} holds.
%\begin{align*}
%\int_\Om \left(\int_\Om \frac{F(y,u_k)}{|x-y|^{\mu}}\right)F(x,u_k)~dx &\leq C\\
%\int_\Om  \left(\int_\Om \frac{F(y,u_k)}{|x-y|^{\mu}}\right)f(x,u_k)u_k & \leq C.
%\end{align*}
Now the proof of \eqref{PS-wk0} follows similarly the proof of \eqref{3.23} (see also equation $(2.20)$ of Lemma $2.4$ in \cite{yang-JDE}). Also, from this we get $u$ forms a weak solution of $(KC)$ using Lemma \ref{kc-ws} and Lemma \ref{wk-sol}. \hfill{\QED}
\end{proof}

 Now we define the associated Nehari manifold as
 \[\mc N = \{u \in W^{1,n}_0(\Om)\setminus \{0\}:\; \langle E^\prime(u),u \rangle=0\}\]
and $l^{**} = \inf_{u\in \mc N}E(u)$.

\begin{Lemma}\label{comp-lem}
If (m3) holds then $l^{*} \leq l^{**}$.
\end{Lemma}
\begin{proof}
Let $u \in \mc N$ and $h:(0,+\infty)\to \mb R$ be defined as $h(t)=E(tu)$. Then
\[h^\prime(t)= {m(\|tu\|^n)}\|u\|^nt^{n-1} - \int_\Om \left(\int_\Om \frac{F(y,tu)}{|x-y|^{\mu}}dy\right)f(x,tu)u~dx.  \]
Since $u$ satisfies $\langle E^\prime(u),u \rangle=0$, we get
\begin{align*}
h^\prime(t) &= \|u\|^{2n}t^{2n-1}\left( \frac{m(\|tu\|^n)}{t^n\|u\|^n}- \frac{m(\|u\|^n)}{\|u\|^n}\right)\\
&\quad + t^{2n-1}\left[\int_\Om\left( \int_\Om \frac{\frac{ F(y,u)f(x,u)}{u^{n-1}(x)}}{|x-y|^{\mu}}dy- \int_\Om \frac{\frac{F(y,tu)f(x,tu)}{(tu(x))^{n-1}t^n}}{|x-y|^{\mu}}dy\right)u^{n}(x)~dx\right].
\end{align*}
\noi
\begin{equation}\label{**}
\textbf{Claim:}\  \text{For any}\  x\in \Omega, t\to tf(x,t)-nF(x,t) \ \text{is increasing on }\R^+.\ \ \ \ \ \ \ \ \ \ \ \ \ \ \ \ \ \ \ \ \ \ \ \ \ \ \ \ \ \ \
\end{equation}
indeed, from (h3), for $0<t_1<t_2$, we have
\begin{align*}
t_1f(x,t_1)-nF(x,t_1) \leq t_1f(x,t_1)-nF(x,t_2)+\frac{f(x,t_2)}{t_2^{n-1}}(t_2^n-t_1^n)\leq t_2f(x,t_2)-nF(x,t_2).
\end{align*}
Using this we get that $tf(x,t)-nF(x,t)\geq 0$ for $t\geq 0$ which implies that $t \to \frac{F(x,tu)}{t^{n}}$ is non-decreasing for $t>0$. Therefore for $0<t<1$ and $x \in \Om$, we get $\frac{F(x,tu)}{t^n} \leq F(x,u)$ and this implies
\begin{align*}
h^\prime(t) &\geq \|u\|^{2n}t^{2n-1}\left( \frac{m(\|tu\|^n)}{\|tu\|^n}- \frac{m(\|u\|^n)}{\|u\|^n}\right)\\
&\quad+ t^{2n-1}\left[\int_\Om \left(\int_\Om\left(F(y,u)- \frac{F(y,tu)}{t^n}\right)~\frac{dy}{|x-y|^{\mu}}\right)\frac{f(x,tu)}{(tu(x))^{n-1}}u^{n}(x)~dx\right].
\end{align*}
This gives that $h^\prime(t)\geq 0$ for $0<t\leq1$ and $h^\prime(t)<0$ for $t>1$. Hence $E(u)= \max_{t\geq 0} E(tu)$. Now we define $g:[0,1] \to W^{1,n}_0(\Om)$ as $g(t)=(t_0u)t$ where $t_0>1$ is such that $E(t_0u)<0$. So $g \in \Gamma$, where $\Gamma$ is as defined in the definition of $l^*$. Therefore we obtain
\[l^* \leq \max_{t\in[0,1]}E(g(t)) \leq \max_{t\geq 0} E(tu)=E(u). \]
Since $u \in \mc N$ is arbitrary, we get $l^* \leq l^{**}$. This completes the proof.\hfill{\QED}
\end{proof}

We recall the following Lemma from \cite{Lions} which is known as the higher integrability Lemma.

\begin{Lemma}\label{Lions-lem}
Let $\{v_k \in W^{1,n}_0(\Om):\; \|v_k\|=1\}$ be a sequence in $W^{1,n}_0(\Om)$ converging weakly to a non zero $v \in W^{1,n}_0(\Om)$. Then for every $p\in \left(1, (1-\|v\|)^{-\frac{1}{n-1}}\right)$,
\[\sup_k \int_\Om \exp \left( p\alpha_n |v_k|^{\frac{n}{n-1}}\right)<+\infty. \]
\end{Lemma}

\noi \textbf{Proof of Theorem \ref{kc-mt-1}:} Let $\{u_k\}$ denotes a Palais Smale sequence at the level $l^*$. Then $(u_k)_{k\in\N}$ can be obtained as a minimizing sequence associated to the variational problem \eqref{*}. Then by Lemma \ref{PS-ws} we know that there exists a $u_0 \in W^{1,n}_0(\Om)$ such that up ta a subsequence
$u_k \rightharpoonup u_0$ weakly in $W^{1,n}_0(\Om)$ as $k \to \infty$. % We know that $u_k \rightharpoonup u_0$ weakly in $W^{1,n}_0(\Om)$ and
%\[\int_\Om \left(\int_\Om \frac{ F(y,u_k)}{|x-y|^{\mu}}dy\right)f(x,u_k)u_k ~dx \leq C \int_\Om f(x,u_k)u_k ~dx.  \]
So if $u_0 \equiv 0$ then using Lemma \ref{PS-ws}, we infer that
\[\int_\Om \left(\int_\Om \frac{ F(y,u_k)}{|x-y|^{\mu}}dy\right)F(x,u_k) ~dx  \to 0\; \text{as}\; k \to \infty.\]
This gives that $ \lim_{k \to \infty}E(u_k) = \frac{1}{n}\lim_{k\to\infty} M(\|u_k\|^n) = l^*$ which implies in the light of Lemma \ref{PS-level} that for large enough $k$
\[M(\|u_k\|^n) < M\left(\left( \frac{2n-\mu}{2n}\alpha_n\right)^{n-1}\right).\] Therefore since $M$ is non decreasing, we get
\[\frac{2n}{2n-\mu}\|u_k\|^{\frac{n}{n-1}} <\alpha_{n}.  \]
Now, this implies that $\sup_{k} \int_{\Omega} f(x,u_k)^q ~dx < +\infty$ for some $q >\frac{2n}{2n-\mu}$ and along with Proposition \ref{HLS}, Theorem \ref{TM-ineq} and the Vitali's convergence theorem,
\[\int_{\Om}\left( \int_\Om\frac{F(y,u_k)}{|x-y|^{\mu}}dy\right)f(x,u_k)u_k~dx \to 0 \;\text{as}\; k \to \infty.\]
Hence $\lim_{k\to \infty}\langle E^\prime(u_k),u_k \rangle=0$ gives $\lim_{k\to \infty}m(\|u_k\|^n)\|u_k\|^n=0$. From (m1) we then obtain $\lim_{k \to \infty} \|u_k\|^n =0$. Thus using Lemma \ref{PS-ws}, it must be that $\lim_{k \to \infty}E(u_k)=0 =l^*$ which contradicts $l^*>0$. Thus $u_0 \not \equiv 0$.\\
\textbf{Claim (1):} $u_0$ is a weak solution of $(KC)$.\\
Before proving this, we show that $u_0>0$ in $\Om$. From Lemma \ref{kc-PS-bdd} we know that $\{u_k\}$ must be bounded. Therefore there exists a constant $\tau >0$ such that up to a subsequence $\|u_k\| \to \tau$ as $k \to \infty$. Since $E^\prime(u_k) \to 0$, again up to a subsequence $|\nabla u_k|^{n-2}\nabla u_k \rightharpoonup |\nabla u_0|^{n-2}\nabla u_0$ weakly in $(L^{\frac{n}{n-1}}(\Om))^n$. Furthermore, by Lemma~\ref{wk-sol} and by Lemma~\ref{kc-ws}
 $$\int_\Om \left(\int_\Om \frac{F(y,u_k)}{|x-y|^{\mu}}dy\right)f(x,u_k) \varphi ~dx \to \int_\Om \left(\int_\Om \frac{F(y,u_0)}{|x-y|^{\mu}}dy\right)f(x,u_0)\varphi~dx$$ and
 \begin{align*}
 m(\tau^n)\int_\Om |\nabla u_0|^{n-2}\nabla u_0\nabla \varphi~dx = \int_\Om \left(\int_\Om \frac{F(y,u_0)}{|x-y|^{\mu}}dy\right)f(x,u_0)\varphi~dx, \; \text{for all}\; \varphi \in W^{1,n}_0(\Om)
 \end{align*}
as $k \to \infty$. In particular, taking $\varphi = u_0^-$ in the above equation we get $m(\tau^n)\|u_0^-\|=0$ which implies together with assumption (m1) that $u_0^-=0$ a.e. in $\Om$. Therefore $u_0\geq 0$ a.e. in $\Om$.\\
  From Theorem \ref{TM-ineq}, we have $f(\cdot,u_0)\in L^q(\Om)$ for $1\leq q<\infty$.
  Also as in \eqref{wk-sol7}, we can similarly get that $\int_\Om \frac{F(y,u_0)}{|x-y|^{\mu}}~dy \in L^\infty(\Om)$. Hence $\left(\int_\Om \frac{F(y,u_0)}{|x-y|^{\mu}}~dy \right)f(x,u_0) \in L^q(\Om)$ for $1 \leq q <\infty$. By elliptic regularity results, we finally get that $u_0 \in L^\infty(\Om)$ and $u_0 \in C^{1,\gamma}(\overline{\Om})$ for some $\gamma \in (0,1)$.  Therefore, $u_0>0$ in $\Om$ follows from the strong maximum principle and $u_0\not\equiv 0$.\\
 Now we claim that
 \begin{equation}\label{kc-mt-11}
 m(\|u_0\|^n)\|u_0\|^n \geq \int_\Om \left(\int_\Om \frac{F(y,u_0)}{|x-y|^{\mu}}~dy \right)f(x,u_0)u_0~dx.
 \end{equation}
Arguing by contradiction, suppose that
\[ m(\|u_0\|^n)\|u_0\|^n <\int_\Om \left(\int_\Om \frac{F(y,u_0)}{|x-y|^{\mu}}~dy \right)f(x,u_0)u_0~dx\]
which implies that $\langle E^\prime(u_0),u_0 \rangle <0$. For $t>0$, using \eqref{**} we have that
\begin{align*}
\langle E^\prime(t u_0),u_0 \rangle  &\geq m(t^n\|u_0\|^n)t^{n-1}\|u_0\|^n -  \frac{1}{n}\int_\Om  \left(\int_\Om \frac{f(y,tu_0)tu_0(y)}{|x-y|^{\mu}}~dy \right)f(x,tu_0)u_0~dx\\
 &\geq m_0t^{n-1}\|u_0\|^n - \frac{1}{n}\int_\Om \left(\int_\Om \frac{f(y,tu_0)tu_0(y)}{|x-y|^{\mu}}~dy \right)f(x,tu_0)u_0~dx.
\end{align*}
Since (h3) gives that
\[\lim_{t\to 0^+} \frac{f(x,t)}{t^\gamma}=0 \; \text{uniformly in}\; x\in \Om,\; \text{for all}\; \gamma \in [0,n-1], \]
we can choose $t>0$ sufficiently small so that $\langle E^\prime(tu_0),u_0 \rangle>0$. Thus there exists a $t_*\in (0,1)$ such that $\langle E^\prime(t_*u_0),u_0 \rangle=0$ that is $t_*u_0 \in \mc N$. So using Lemma \ref{comp-lem}, $(m3)^\prime$ and \eqref{**} we get
\begin{align*}
l^* \leq l^{**} &\leq E(t_*u_0) = E(t_*u_0) - \frac{1}{2n}\langle E^\prime(t_*u_0),u_0 \rangle\\
& = \frac{M(\|t_*u_0\|^n)}{n} -\frac12 \int_\Om \left(\int_\Om \frac{ F(y,t_*u_0)}{|x-y|^{\mu}}dy\right)F(x,t_*u_0)~dx -\frac{1}{2n}m(\|t_*u_0\|^n)\|t_*u_0\|^n\\
 & \quad+ \frac{1}{2n}\int_\Om \left(\int_\Om \frac{F(y,t_*u_0)}{|x-y|^{\mu}}dy\right)f(x,t_*u_0)t_*u_0~dx\\
& < \frac{M(\|u_0\|^n)}{n}  -\frac{1}{2n}m(\|u_0\|^n)\|u_0\|^n \\
& \quad + \frac{1}{2n}\int_\Om \left(\int_\Om \frac{ F(y,t_*u_0)}{|x-y|^{\mu}}dy\right)(f(x,t_*u_0)t_*u_0 -nF(x,t_*u_0))~dx\\
& \leq  \frac{M(\|u_0\|^n)}{n}  -\frac{1}{2n}m(\|u_0\|^n)\|u_0\|^n  + \frac{1}{2n}\int_\Om \left(\int_\Om \frac{ F(y,u_0)}{|x-y|^{\mu}}dy\right)(f(x,u_0)u_0 -nF(x,u_0))\\
& \leq \liminf_{k \to \infty} \frac{M(\|u_k\|^n)}{n}  -\frac{1}{2n}m(\|u_k\|^n)\|u_k\|^n \\
& + \frac{1}{2n}\int_\Om \left(\int_\Om \frac{ F(y,u_k)}{|x-y|^{\mu}}dy\right)(f(x,u_k)u_k -nF(x,u_k))~dx\\
& = \liminf_{k \to \infty} \left( E(u_k) - \frac{1}{2n}\langle E^\prime(u_k),u_k \rangle\right) = l^*.
\end{align*}
This gives a contradiction, that is \eqref{kc-mt-11} holds true. \\
\textbf{Claim (2):} $E(u_0)= l^*$.\\
From Lemma \ref{PS-ws} we know that
\[\int_\Om \left(\int_\Om \frac{ F(y,u_k)}{|x-y|^{\mu}}dy\right)F(x,u_k)~dx \to \int_\Om \left(\int_\Om \frac{ F(y,u_0)}{|x-y|^{\mu}}dy\right)F(x,u_0)~dx. \]
Using this and the weakly lower semicontinuity of norms in $\lim_{k \to \infty}E(u_k)= l^*$, we obtain $E(u_0) \leq l^*$. If $E(u_0)< l^*$ then it must be
\[\lim_{k \to \infty} M(\|u_k\|^n) > M(\|u_0\|^n)\]
which implies that $\lim_{k \to \infty}\|u_k\|^n > \|u_0\|^n$, since $M$ is continuous and increasing. From this we get
\[\tau^n > \|u_0\|^n.\]
Moreover we get
\begin{equation}\label{kc-mt-12}
M(\tau^n) = n \left( l^* + \frac12 \int_\Om \left(\int_\Om \frac{ F(y,u_0)}{|x-y|^{\mu}}dy\right)F(x,u_0)~dx\right).
\end{equation}
Now we define the sequence $v_k = \frac{u_k}{\|u_k\|}$ and $v_0 = \frac{u_0}{\tau}$ then $v_k \rightharpoonup v_0$ weakly in $W^{1,n}_0(\Om)$ and $\|v_0\|<1$. From Lemma \ref{Lions-lem} we have that
\begin{equation}\label{kc-mt-13}
\sup_{ k \in \mb N} \int_\Om \exp \left( p|v_k|^{\frac{n}{n-1}}\right)<+\infty,\; \text{for} \; 1<p<\frac{\alpha_n}{(1-\|v_0\|)^{\frac{1}{n-1}}}.
\end{equation}
Also from $(m3)^\prime$, Claim (1) and Lemma \ref{comp-lem} we obtain
\[E(u_0)= \frac{M(\|u_0\|^n)}{n}- \frac{m(\|u_0\|^n)\|u_0\|^n}{2n}+ \frac{1}{2n}\int_\Om \left(\int_\Om \frac{ F(y,u_0)}{|x-y|^{\mu}}dy\right)(f(x,u_0)u_0-nF(x,u_0))~dx\geq 0.\]
Using this with \eqref{kc-mt-12} we get that
\begin{align*}
M(\tau^n) = nl^* - n E(u_0) +M(\|u_0\|^n) < M\left(\left( \frac{2n-\mu}{2n}\alpha_n\right)^{n-1}\right) + M(\|u_0\|^n)
\end{align*}
which implies together with (m1) that
\begin{equation*}
\tau^n < \frac{\alpha_n^{n-1}\left( \frac{2n-\mu}{2n}\right)^{n-1}}{1-\|v_0\|^n}.
\end{equation*}
Thus it is possible to find a $\tau_* > 0$ such that for $k\in \mb N$ large enough
\[\|u_k\|^{\frac{n}{n-1}} < \tau_* < \frac{\alpha_n\left( \frac{2n-\mu}{2n}\right)}{(1-\|v_0\|^n)^{\frac{1}{n-1}}}.\]
%< \frac{\alpha_n}{(1-\|v_0\|^n)^{\frac{1}{n}}}.\]
Then we choose a $q>1$ but close to $1$ such that
\[\frac{2n}{2n-\mu} q \|u_k\|^{\frac{n}{n-1}} \leq \frac{2n}{2n-\mu}\tau_* < \frac{\alpha_{n}}{(1-\|v_0\|^n)^{\frac{1}{n-1}}}.
 \]
Therefore from \eqref{kc-mt-13} we conclude that
\begin{equation}\label{kc-mt-15}
 \int_\Om \exp\left(\frac{2n}{2n-\mu}q|u_k|^{\frac{n}{n-1}}\right) \leq C
 \end{equation}
for some constant $C>0$. Using \eqref{kc-mt-15}
\[\int_\Om \left(\int_\Om \frac{ F(y,u_k)}{|x-y|^{\mu}}dy\right)f(x,u_k)u_k~dx \to \int_\Om \left(\int_\Om \frac{ F(y,u_0)}{|x-y|^{\mu}}dy\right)f(x,u_0)u_0~dx. \]
We conclude that $\|u_k\| \to \|u_0\|$ and we get a contradiction and claim (2) is proved. Now, from Claims (1) and (2), the proof of Theorem~\ref{kc-mt-1} follows.
 \hfill{\QED}
\section{The Nehari Manifold method for $(\mathcal{P}_{\la, M})$}
The energy functional $\mathcal{J}_{\la, M} : W_0^{1,n}(\Omega) \longrightarrow \mathbb{R}$ associated to the problem $\mathcal{P}_{\la, M}$ is defined as
\begin{align*}
\mathcal{J}_{\la, M}(u)= \dfrac{1}{n} M(\|u\|^n) - \dfrac{\la}{q+1} \int_{\Omega} h(x) |u|^{q+1}~dx - \dfrac{1}{2} \int_{\Omega} (|x|^{-\mu}* F(u)) F(u)~dx
\end{align*}
where $F$ and $M$ are primitive of $f$ and $m$ respectively and $f(s)= s|s|^p exp(|s|^{\beta}).$
\begin{Definition}
A function $u \in W_0^{1,n}(\Omega)$ is said to be weak solution of \ $\mathcal{P}_{\la, M}$ if \ $\forall \ \phi \in W_0^{1,n}(\Omega)$ we have
\begin{align*}
m(\|u\|^n)\int_{\Omega} |\nabla u|^{n-2} \nabla u. \nabla \phi ~dx= \la \int_{\Omega} h(x) u^{q-1} u \phi ~dx + \int_{\Omega} (|x|^{-\mu}* F(u)) f(u) \phi ~dx.
\end{align*}
\end{Definition}
We observe that $\mathcal{J}_{\la, M}$ is only bounded below on suitable subsets of $W_0^{1,n}(\Omega)$. In order to prove the existence of weak solutions to  $(\mathcal{P}_{\la, M})$, we establish the existence of minimizers of $\mathcal{J}_{\la, M}$ under the natural constraint of the Nehari Manifold:
\begin{equation*}
N_{\la, M} := \{u \in W_0^{1,n}(\Omega)|\  \langle \mathcal{J}^{'}_{\la,M} (u),u \rangle =0 \}
\end{equation*}
where $\langle . \ , . \rangle$ denotes the duality between $W_0^{1,n}(\Omega)$ and $W^{-1,n}(\Omega).$ Therefore, $u \in N_{\la, M}$ if and only if
$$ \|u\|^n \  m(\|u\|^n)- \la \int_{\Omega} h(x) u^{q+1} ~dx - \int_{\Omega} (|x|^{-\mu}* F(u)) f(u) u ~dx =0.$$
\begin{Remark}
We noticed that $N_{\la, M}$ contains every solution of $(\mathcal{P}_{\la,M}).$
\end{Remark}
For $u \in W_0^{1,n}(\Omega)$, we define the fiber map $\Phi_{u,M} : \mathbb{R}^{+} \rightarrow \mathbb{R}$ as
\begin{equation*}
\Phi_{u,M} (t)= \mathcal{J}_{\la, M}(tu)= \frac{M(\|tu\|^n)}{n}  - \frac{\la}{q+1} \int_{\Omega} h(x) |tu|^{q+1} ~dx - \frac{1}{2} \int_{\Omega} (|x|^{-\mu}*F(tu)) F(tu) ~dx,
\end{equation*}
\begin{equation*}
\Phi^{'}_{u,M} (t) = t^{n-1} \|u\|^n m(\|tu\|^n) - \la t^{q} \int_{\Omega} h(x) |u|^{q+1}~dx - \int_{\Omega} (|x|^{-\mu}* F(tu)) f(tu) u ~dx
\end{equation*}
and
\begin{align*}
      \Phi^{''}_{u,M} (t)&= n t^{2n-2} \|u\|^{2n} m'(\|tu\|^n)+ (n-1) t^{n-2} \|u\|^n m(\|tu\|
      ^n) - \la q t^{q-1} \int_{\Omega} h(x) |u|^{q+1}~dx \\
      &- \int_{\Omega} (|x|^{-\mu} *f(tu).u)f(tu)u ~dx - \int_{\Omega} (|x|^{-\mu}* F(tu))f'(tu) u^2 ~dx.
 \end{align*}
 The Nehari Manifold is closely related to the the maps $\Phi_{u, M}$ by the relation $tu \in N_{\la, M}$ iff $\Phi_{u,M}^{'}(t)=0$. In particular, $ u \in N_{\la,M}$ iff $\Phi_{u,M}^{'}(1)=0$. So we study the geometry of the energy functional on the following components of the Nehari Manifold:
\begin{equation*}
N^+_{\la, M}:= \{u \in N_{\la,M} : \Phi_{u,M}^{''}(1)>0\} =\{tu \in W_0^{1,n}(\Omega) : \Phi_{u,M}^{'}(t)=0, \Phi_{u,M}^{''}(t)>0\},
\end{equation*}
\begin{equation*}
N^-_{\la, M}:= \{u \in N_{\la,M} : \Phi_{u,M}^{''}(1)<0\} =\{tu \in W_0^{1,n}(\Omega) : \Phi_{u,M}^{'}(t)=0, \Phi_{u,M}^{''}(t)<0\},
\end{equation*}
\begin{equation*}
N^0_{\la, M}:= \{u \in N_{\la,M} : \Phi_{u,M}^{''}(1)=0\} =\{tu \in W_0^{1,n}(\Omega) : \Phi_{u,M}^{'}(t)=0, \Phi_{u,M}^{''}(t)=0\}.
\end{equation*}
We also define $H(u)=\int_{\Omega} h|u|^{q+1}~dx$ and study the behaviour of fibering maps $\Phi_{u,M}$ according to the sign of $H(u).$ Let
\begin{equation*}
H^{+}:= \{u \in W_0^{1,n}(\Omega) : H(u) >0 \},
\end{equation*}
\begin{equation*}
H^{-}_0:= \{u \in W_0^{1,n}(\Omega) : H(u) \leq 0 \}.
\end{equation*}
%\begin{equation*}
%H_{0}:= \{u \in W_0^{1,n}(\Omega) : H(u) =0 \}.
%\end{equation*}
\subsection{Analysis of Fiber Maps}
Here we analyze accurately the geometry of the energy functional on the Nehari manifold. We split the study in different cases.\\
Case 1: $ u \in H^-_0$\\
Define $\psi : \mathbb{R}^+ \rightarrow \mathbb{R}$ such that
\begin{equation*}
\psi_u(t)= t^{n-1-q} m(\|tu\|^n) \|u\|^n - t^{-q} \int_{\Omega} (|x|^{-\mu}* F(tu)) f(tu)u ~dx.
\end{equation*}
Since
\begin{align*}
\Phi^{'}_{u,M}(t)&= t^{n-1} \|u\|^n m(\|tu\|^n) - \la t^q \int_{\Omega} h(x) |u|^{q+1} ~dx - \int_{\Omega} (|x|^{-\mu}* F(tu)) f(tu)u ~dx\\
&= t^q (\psi_u(t) - \la \int_{\Omega} h(x) |u|^{q+1}~dx),
\end{align*}
$tu \in N_{\la,M}$ iff $t>0$ is a solution of $ \psi_u(t)= \la \int_{\Omega} h(x) |u|^{q+1}.$
\begin{equation}\label{KC21}
\begin{split}
 \psi^{'}_u(t)&= (n-1-q) t^{n-2-q} m(\|tu\|^n) \|u\|^n + n  t^{2n-2-q} m{'}(\|tu\|^n) \|u\|^{2n}\\
 &+ \frac{q}{t^{q+1}} \int_{\Omega} (|x|^{-\mu}* F(tu)) f(tu).u ~dx - t^{-q} \bigg[ \int_{\Omega} (|x|^{-\mu}* f(tu).u) f(tu).u ~dx \\
 &+ \int_{\Omega} (|x|^{-\mu}* F(tu)) f^{'}(tu).u^2 ~dx \bigg].
\end{split}
\end{equation}
Due to the exponential growth of $f$, for large $t$ we have $\psi_u^{'}(t) <0$ and since $u \in H^-_0$, there exists $t^*>0$ such that $\psi_u(t^*)= \la \int_{\Omega} h(x) |u|^{q+1}$, {\it i.e.} $t^*u \in N_{\la, M}$.\\
 If there exists an another point $t_1$ such that $t^*< t_1$ and $\psi_u(t_1)= \la \int_{\Omega} h(x) |u|^{q+1} \leq 0$, {\it i.e.}
 \begin{equation}\label{role}
 t_1^{n-1-q} (a t_1^n \|u\|^n +b) \|u\|^n \leq t_1^{-q} \int_{\Omega} (|x|^{-\mu}*F(t_1 u)) f(t_1 u)u~dx
 \end{equation}
 and $\psi_u'(t_1) \geq 0$. Then by using $f'(t_1u) t_1u > (p+1) f(t_1u)$ and $p> 2n-2-q$ we obtain from \eqref{role},
  \begin{align*}
  \psi'_u(t_1)& <(2n-1-q)\left[t_1^{n-2-q} (a t_1^n \|u\|^n +b) \|u\|^n - t_1^{-q-1} \int_{\Omega} (|x|^{-\mu}*F(t_1 u)) f(t_1 u)u~dx \right] \leq 0.
  \end{align*}
Therefore $\psi_u'(t_1)<0$ which yields a contradiction. Therefore there exists a unique $t^*$ such that $\psi_u(t^*)= \la \int_{\Omega} h(x) |u|^{q+1}~dx$. Also for $0<t<t^*$, $\Phi^{'}_{u,M}(t)= t^q(\psi_u(t)-\la \int_{\Omega} h(x) |u|^{q+1} ~dx) > 0$. Consequently, $\Phi_{u,M}$ is increasing in $(0,t^*)$ and decreasing on $(t^*, \infty).$ Therefore there exists a unique critical point of $\Phi_{u, M}$ which is also a global maximum point. Furthermore, since $\psi'_u(t) = \displaystyle\frac{\left(t \phi''_{u,M}(t)-q \phi'_{u,M}(t)\right)}{t^q}$, therefore $t^* u \in N^-_{\la,M}.$\\ \\
Case 2: $ u \in H^+$\\
In this case, we establish that there exists $\la_0 >0$ and a $t_*$ such that for $\la \in (0, \la_0)$, $\Phi_{u}$ has exactly two critical points $t_1(u)$ and $t_2(u)$ such that $t_1(u) < t_*(u) <t_2(u)$ where $t_1(u)$ is local minimum point and $t_2(u)$ is local maximum point. To prove this case, we need the analysis performed in the next subsection.
\subsection{Preliminary Results for Case-2}
  For $0\not\equiv u \in H^+$, we have that $\psi_u(t) \rightarrow -\infty$ as $ t \rightarrow \infty$ and for small $t>0,$ $\psi_u(t)>0$. Then there exists at least a point of maximum of $\psi_u(t)$, say $t_*$, and $\psi_u^{'}(t_*)=0$, {\it i.e.}
 \begin{align*}
 (2n-1-q) t_*^{2n-2-q} a \|u\|^{2n} &+(n-1-q) t_*^{n-2-q} b \|u\|^n + \dfrac{q}{t_*^{q+1}} \int_{\Omega}(|x|^{-\mu}*F(t_*u))f(t_*u)u ~dx\\
 &=t_*^{-q}\bigg[\int_{\Omega}(|x|^{-\mu}*F(t_*u))f^{'}(t_*u) u^2 ~dx +\int_{\Omega} (|x|^{-\mu}*f(t_*u)u)f(t_*u).u ~dx\bigg].
\end{align*}
This implies that
  \begin{align*}
 (2n-1-q)a \|t_*u\|^{2n} &+(n-1-q) b \|t_*u\|^n + q\int_{\Omega}(|x|^{-\mu}*F(t_*u))f(t_*u)t_*u ~dx\\
 &=\int_{\Omega}(|x|^{-\mu}*F(t_*u))f^{'}(t_*u) (t_*u)^2 ~dx +\int_{\Omega} (|x|^{-\mu}*f(t_*u)t_*u)f(t_*u)t_*u ~dx.
\end{align*}
Then we have
\begin{align*}
2 \sqrt{(2n-1-q)a \|t_*u\|^{2n}b (n-1-q)\|t_*u\|^n} \leq B(t_*u)
\end{align*} from which it follows
\begin{align*}
\|t_*u\|^{3n/2} \leq \frac{B(t_*u)}{2\sqrt{(2n-1-q)(n-1-q)ab}}
\end{align*}
where $B(u)=\int_{\Omega}(|x|^{-\mu}*F(u))f^{'}(u) u^2 +\int_{\Omega} (|x|^{-\mu}*f(u)u)f(u) u ~dx$.
Using $\psi^{'}_u(t_*)=0$, we replace the value of $a\|tu\|^{2n}$ in the definition of $\psi_u(t)$ to obtain \\
\begin{equation}\label{KC1}
\psi_u(t_*)= \frac{1}{(2n-1-q) t_*^{q+1}} \bigg[ B(t_*u)- (2n-1) \int_{\Omega}(|x|^{-\mu}*F(t_*u)) f(t_*u)t_*u ~dx + nb \|t_* u\|^n \bigg].
\end{equation}
Now we prove the following result and establish the proof in various steps.
 \begin{Lemma}\label{inf}
 Let $$\Gamma := \left\{ u \in W_0^{1,n}(\Omega): \|u\|^{3n/2} \leq \frac{B(u)}{2\sqrt{(2n-1-q)(n-1-q)ab}}\right\}$$ where $B(u)=\int_{\Omega}(|x|^{-\mu}*F(u))f^{'}(u) (u)^2 +\int_{\Omega} (|x|^{-\mu}*f(u)u)f(u)u ~dx.$ Then there exists a $\la_0>0$ such that for every $ \la \in (0, \la_0)$, $\Gamma_0 >0$ holds where
\begin{equation}\label{imp}
\Gamma_0 := \inf_{u \in \Gamma\backslash \{0\} \cap H^+}\bigg[ B(u)- (2n-1) \int_{\Omega}(|x|^{-\mu}*F(u)) f(u).u ~dx + nb \|u\|^n - \la (2n-1-q) H(u) \bigg].
\end{equation}
 \end{Lemma}
 \begin{proof}
 Step 1:\ Claim: $\inf_{u \in \Gamma\backslash\{0\} \cap H^+} \|u\|>0.$\\
Let us suppose that it doesn't hold then there exists a sequence $\{u_k\} \subset \Gamma\backslash \{0\} \cap H^+$ such that $\|u_k\| \rightarrow 0$ and $\|u_k\|^{3n/2} \leq \frac{B(u_k)}{2\sqrt{(2n-1-q)(n-1-q)ab}}, \forall \ k$.
 Then by the Hardy-Littlewood-Sobolev inequality we have,
 \begin{align*}
 B(u_k)&=\int_{\Omega}(|x|^{-\mu}*F(u_k))f^{'}(u_k) u_k^2 ~dx +\int_{\Omega} (|x|^{-\mu}*f(u_k)u_k)f(u_k)u_k ~dx\\
 &\leq C(n, \mu) \left(\|f(u_k)u_k\|^2_{L^{2n/(2n-\mu)}(\Omega)} + \|F(u_k)\|_{L^{2n/(2n-\mu)}(\Omega)} \|f^{'}(u_k) (u_k)^2\|_{L^{2n/(2n-\mu)}(\Omega)}\right).
 \end{align*}
 Since $f(u)= u |u|^p exp(|u|^{\beta})$ and $f^{'}(u)= ((p+1)+\beta |u|^{\beta})|u|^p exp(|u|^{\beta})$
  then, we have
 \begin{align*}
 |B(u_k)|\leq C(n, \mu)\bigg(\int_{\Omega} (|u_k|^{p+2} &exp(|u_k|^{\beta}))^{\frac{2n}{2n-\mu}}~dx\bigg)^{\frac{2n-\mu}{n}} +  C(n, \mu) \left(\int_{\Omega} (F(u_k))^{\frac{2n}{2n-\mu}}~dx
 \right)^{\frac{2n-\mu}{2n}} \\
 &.\left(\int_{\Omega}(((p+1)+\beta |u_k|^{\beta})|u_k|^{p+2} exp(|u_k|^{\beta}))^{\frac{2n}{2n-\mu}}~dx\right)^{\frac{2n-\mu}{2n}}.
 \end{align*}
 Then using $F(t) \leq tf(t)$ and by H\"older's inequality, we obtain
 \begin{align*}
 |B(&u_k)|\leq C_1\bigg(\int_{\Omega} |u_k|^{\frac{2n \alpha'(p+2)}{2n-\mu}}~dx\bigg)^{\frac{2n-\mu}{n\alpha'}}. \bigg(\int_{\Omega} exp\bigg(|u_k|^{\beta}\frac{2n \alpha}{2n-\mu}\bigg)~dx\bigg)^{\frac{2n-\mu}{n\alpha}}\\
 &+  C_2 \bigg(\int_{\Omega} |u_k|^{\frac{2n \alpha'(p+2)}{2n-\mu}}~dx\bigg)^{\frac{2n-\mu}{2n\alpha'}}. \bigg(\int_{\Omega} exp\bigg(|u_k|^{\beta} \frac{2n \alpha}{2n-\mu}\bigg)~dx\bigg)^{\frac{2n-\mu}{2n\alpha}} \times\\
 &\left[\bigg(\int_{\Omega} |u_k|^{\frac{2n\alpha'(p+2)}{2n-\mu}}~dx \bigg)^{\frac{2n-\mu}{2n\alpha'}}. \bigg(\int_{\Omega} exp\bigg(|u_k|^{\beta} \frac{2n \alpha}{2n-\mu}\bigg)~dx\bigg)^{\frac{2n-\mu}{2n\alpha}}\right. \\
 &+ \left.\bigg(\int_{\Omega} |u_k|^{\frac{2n\alpha'(p+\beta+2)}{2n-\mu}}~dx\bigg)^{\frac{2n-\mu}{2n\alpha'}}. \bigg(\int_{\Omega} exp\bigg(|u_k|^{\beta} \frac{2n \alpha}{2n-\mu}\bigg)~dx\bigg)^{\frac{2n-\mu}{2n\alpha}}\right].
 \end{align*}
  Let $\alpha$ be such that $2n \alpha /(2n-\mu))\|u_k\|^\beta \leq\alpha_n $ and  $v_k= \frac{u_k}{||u_k||}$, then by the Trudinger-Moser inequality we obtain
 \begin{align*}
 |B(u_k)|&\leq C_1\bigg(\int_{\Omega} |u_k|^{\frac{2n \alpha'(p+2)}{2n-\mu}}~dx\bigg)^{\frac{2n-\mu}{n\alpha'}}. \bigg(\sup_{\|v_k\| \leq 1}\int_{\Omega} exp(|v_k|^{\beta}\alpha_n)~dx \bigg)^{\frac{2n-\mu}{n\alpha}}\\
 &+  C_2 \bigg(\int_{\Omega} |u_k|^{\frac{2n \alpha'(p+2)}{2n-\mu}}~dx\bigg)^{\frac{2n-\mu}{2n\alpha'}}.\bigg(\sup_{\|v_k\| \leq 1}\int_{\Omega} exp(|v_k|^{\beta}\alpha_n)~dx \bigg)^{\frac{2n-\mu}{n\alpha}} \times\\
 &\left[\bigg(\int_{\Omega} |u_k|^{\frac{2n\alpha'(p+2)}{2n-\mu}}~dx\bigg)^{\frac{2n-\mu}{2n\alpha'}} + \bigg(\int_{\Omega} |u_k|^{\frac{2n\alpha'(p+\beta+2)}{2n-\mu}}~dx\bigg)^{\frac{2n-\mu}{2n\alpha'}}\right].
 \end{align*}
Using Sobolev embedding, it implies that
 \begin{align*}
 |B(u_k)|&\leq C_1(n,k, \beta, \mu)( \|u_k\|^{2(p+2)} +\|u_k\|^{(p+2)}(\|u_k\|^{(p+2)}+\|u_k\|^{(p+\beta+2)}))\leq C \|u_k\|^{(2p+4)}+ \|u_k\|^{(2p+\beta+4)}.
 \end{align*}
Hence using $u_k \in \Gamma\backslash\{0\}$ and by Sobolev embedding theorem, we get
$1 \leq C (\|u_k\|^{(2p+4- \frac{3n}{2})}+\|u_k\|^{(2p+\beta+4- \frac{3n}{2})}$ and $2p+4- \frac{3n}{2} >0$ which is a contradiction as $\|u_k\| \to 0$ as $k \to \infty$. Therefore we have $\inf_{u \in \Gamma\backslash\{0\} \cap H^+} \|u\|>0.$\\
Step 2: Claim: $0< \inf_{u \in \Gamma\backslash\{0\}\cap H^+} \left\{\displaystyle\int_{\Omega} (|x|^{-\mu}*f(u)u)(p+2-2n+\beta |u|^{\beta})exp(|u|^\beta) |u|^{p+2} ~dx \right\} $.\\
Since $F(s) \leq \frac{f(s)s}{p+2}$ then by the definition of $\Gamma$ and from Step 1, we obtain $0< \inf_{u \in \Gamma\backslash\{0\}\cap H^+} B(u)$ {\it i.e.} \\
\begin{align*}
0&< \inf_{u \in \Gamma\backslash\{0\}\cap H^+} \left\{\int_{\Omega} (|x|^{-\mu}*F(u))f'(u)u^2 +\int_{\Omega} (|x|^{-\mu}*f(u)u)f(u)u\right\}\\
&\leq \inf_{u \in \Gamma\backslash\{0\}\cap H^+} \left\{\int_{\Omega} (|x|^{-\mu}*f(u)u)\bigg(f(u).u+ f'(u)\frac{u^2}{p+2}\bigg)\right\}\\
&=\inf_{u \in \Gamma\backslash\{0\}\cap H^+} \left\{\int_{\Omega} (|x|^{-\mu}*f(u)u)|u|^{p+2} exp(|u|^\beta)\bigg(1+ \frac{(p+1)+\beta |u|^{\beta}}{p+2}\bigg)\right\}.
\end{align*}
Since $p+2-2n>0$, we infer $$0< \inf_{u \in \Gamma\backslash\{0\}\cap H^+} \{\int_{\Omega} (|x|^{-\mu}*f(u)u)(p+2-2n+\beta |u|^{\beta})exp(|u|^\beta) |u|^{p+2} ~dx \}.$$

Step 3: Claim: $\Gamma_0>0$. First,
\begin{equation*}
\int_{\Omega} h(x) |u|^{q+1} \leq \bigg(\int_{\Omega} |h(x)|^\gamma\bigg)^{1/\gamma} \bigg(|u|^{(1+q)\gamma'}\bigg)^{1/{\gamma'}}\leq l \|u\|^{q+1}.
\end{equation*}
where $l = \|h\|_{L^\gamma(\Omega)}$. Choosing
$$\la < \frac{bn}{(2n-1-q)l} M_0 :=\la_0 $$ where $M_0= \inf_{u \in \Gamma \backslash \{0\} \cap H^+} \|u\|^{n-1-q} > 0$, we get that $\la l(2n-1-q) \|u\|^{1+q} < \ nb \|u\|^n$ for any $u \in \Gamma \backslash \{0\} \cap H^+$ . Then for $u \in \Gamma\backslash\{0\}\cap H^+$ and $p+1 > 2n-1$,
\begin{align*}
B(u)+nb\|u\|^n- &(2n-1)\int_{\Omega} (|x|^{-\mu}*F(u)) f(u).u -\la (2n-1-q) H(u) \geq\\
&\int_{\Omega} (|x|^{-\mu}*F(u))(f'(u)u^2-(2n-1) f(u).u)+ \int_{\Omega} (|x|^{-\mu}*f(u).u)f(u).u ~dx \\
&+ nb \|u\|^n - (2n-1-q) \la H(u) > 0.
\end{align*}
Therefore $\Gamma_0 >0$.
\hfill{\QED}
\end{proof}

Now we establish the claim made in Case-2. We notice from Lemma  \ref{inf} and Equation \eqref{KC1} that for $u \in H^+ \backslash \{0\}$, there exists a $t_*>0$, local maximum of $\psi_u$ verifying $\psi_u(t_*)-\la H(u)>0$ since $t_*u \in \Gamma\setminus\{0\} \cap H^+$. From $\psi_u(0)=0$, $\psi_u(t_*)> \la H(u) >0$ and $\lim_{t \rightarrow \infty} \psi_u(t)= -\infty$, there exists $t_1=t_1(u) < t_* < t_2(u)=t_2$ such that $\psi_u(t_1)= \la \int_{\Omega} h(x)|u|^{q+1} ~dx=\psi_u(t_2)$ with $\psi_u'(t_1)>0,\psi_u'(t_2)<0$. Therefore, $t_1u \in N^+_{\la, M}$ and  $t_2u \in N^-_{\la, M}$. We now prove that $t_1 u \in N^+_{\la, M}$ and $t_2 u \in N^-_{\la, M}$ are unique. If not then there exists $t_3 u \in N^+_{\la, M}$ and $t_{**}$ such that $t_2 < t_{**} < t_3$  and $\psi_u'(t_{**})=0$ and $\psi_u(t_{**}) <  \la H(u).$ But our Lemma \ref{inf} induces that if $\psi_u'(t_{**})=0$ then  $\psi_u(t_{**}) > \la H(u)$ which is a contradiction.

In the sequel we will denote $t_*$ the smallest critical point of $\psi_u.$
\begin{Lemma}
If $\la \in (0,\la_0)$ then $N^0_{\la, M}= \{0\}.$
\end{Lemma}
\begin{proof}
Suppose $u \not\equiv 0$ and $u \in N^0_{\la,M}.$ Then $\Phi_{u, M}^{'}(1)=0$ and $\Phi_{u, M}^{''}(1)=0$, {\it i.e.}
 \begin{align}\label{KC2}
a\|u\|^{2n}+ b\|u\|^{n}= \la H(u)+ \int_{\Omega} (|x|^{-\mu}*F(u)) f(u)u ~dx\mbox{ and}
\end{align}
\begin{align}\label{kc3}
(2n-1) a \|u\|^{2n} +(n-1) b \|u\|^n= \la q H(u)+B(u).
\end{align}
Let $u \in H^+ \cap N_{\la, M}^0,$ then from \eqref{KC2} and \eqref{kc3} (by replacing the value $\la H(u)$), we obtain
$$2 \sqrt{(2n-1-q)(n-1-q)ab \|u\|^{3n}}  \leq B(u)$$
which implies $u \in \Gamma \backslash \{0\} \cap H^+.$
Again from \eqref{KC2}, \eqref{kc3} and substituting the value of $a \|u\|^{2n}$, we obtain
\begin{equation*}
B(u)- (2n-1)\int_{\Omega} (|x|^{\mu}*F(u))f(u).u + nb \|u\|^n - \la (2n-1-q)H(u)=0
\end{equation*} which contradicts Lemma \ref{inf}.
If $u \in H^-_0 \cap N_{\la, M}^0$ and from Case-1, $"1"$ is the only critical point of $\Phi_{u,M}$ and $\Phi^{''}_{u,M}(1)=0.$ But $u \in H^-_0$ implies that $\psi'_u(1)<0$ and then $\phi''_{u,M}(1) <0$ which is a contradiction and  the lemma is proved.
\hfill{\QED}
\end{proof}
\subsection{Existence of weak solutions to $(\mathcal{P}_{\la, M})$}
In this section we prove that $\mathcal{J}_{\la,M}$ is bounded below on $N_{\la,M}$ and achieves its minimum. Define $\theta =\inf_{u \in N_{\la,M}} \mathcal{J}_{\la,M}(u).$
\begin{Theorem}\label{bdd}
$\mathcal{J}_{\la,M}(u)$ is bounded below and coercive on $N_{\la,M}$ such that $\theta \geq - C(q,n,b) \la^{\frac{n}{n-q-1}}.$
\end{Theorem}
\begin{proof}
Let $u \in N_{\la, M}.$ Then,
\begin{align*}
\mathcal{J}_{\la,M}(u)&= \frac{1}{n}\bigg[\frac{a}{2} \|u\|^{2n} + b \|u\|^n\bigg]- \frac{\la}{q+1} H(u) - \frac{1}{2} \int_{\Omega} (|x|^{-\mu}* F(u))F(u)~dx\\
&= \frac{1}{n}\bigg[\frac{a}{2} \|u\|^{2n} + b \|u\|^n\bigg]- \frac{\la}{q+1} H(u) - \frac{1}{2} \int_{\Omega} (|x|^{-\mu}* F(u))F(u)~dx - \frac{1}{p+2}\bigg[a\|u\|^{2n}+ b\|u\|^{n}\\
&\ \ \ \ \ \ \ \ - \la H(u)- \int_{\Omega} (|x|^{-\mu}*F(u)) f(u)u ~dx \bigg]\\
&=a \|u\|^{2n}\bigg(\frac{p+2-2n}{2n(p+2)}\bigg)+ b \|u\|^{n}\bigg(\frac{p+2-n}{n(p+2)}\bigg)- \la\bigg(\frac{p+1-q}{(1+q)(p+2)}\bigg)H(u)\\
&\ \ \ \ \ \ \ \ -\frac{1}{2} \int_{\Omega} (|x|^{-\mu}*F(u))\left(F(u)-\frac{2f(u)u}{p+2}\right)~dx.
\end{align*}
Since $0 \leq F(u) \leq \frac{2}{p+2} f(u).u$ and $H(u)\leq l \|u\|^{q+1}$. Then by the Sobolev inequality we obtain
\begin{align*}
\mathcal{J}_{\la,M}(u) \geq a \|u\|^{2n}\bigg(\frac{p+2-2n}{2n(p+2)}\bigg)+ b \|u\|^{n}\bigg(\frac{p+2-n}{n(p+2)}\bigg)- \la l \bigg(\frac{p+1-q}{(1+q)(p+2)}\bigg)\|u\|^{q+1}.
\end{align*}
Therefore since $q< n-1$, $\,\mathcal{J}_{\la,M}$ is coercive on $N_{\la, M}$, {\it i.e.}
$\mathcal{J}_{\la,M}(u) \rightarrow \infty $ as $\|u\| \rightarrow \infty.$\\
For $u \in N_{\la, M}$ we have also,
\begin{align*}
\mathcal{J}_{\la,M} (u)&= \frac{b}{n} \|u\|^n -\frac{\la}{q+1} H(u)- \frac{1}{2} \int_{\Omega} (|x|^{-\mu}*F(u))F(u)~dx \\
&\ \ \ \ \ + \frac{1}{2n}\bigg(\la H(u) + \int_{\Omega} (|x|^{-\mu}*F(u))f(u)u ~dx - b\|u\|^n \bigg)\\
&= \frac{1}{2n} b \|u\|^n- \la \bigg(\frac{1}{q+1}-\frac{1}{2n}\bigg) H(u)+\frac{1}{2} \bigg( \int_{\Omega} (|x|^{-\mu}*F(u))\bigg(\frac{f(u)u}{n} -F(u)\bigg)~dx\bigg)\\
& \geq \frac{1}{2n} b \|u\|^n- \la \bigg(\frac{1}{q+1}-\frac{1}{2n}\bigg) H(u)
\end{align*}
since $\left(\frac{f(u)u}{n} -F(u)\right)\geq 0.$ Then for $u \in H_0^-$, we get $\mathcal{J}_{\la,M}(u) \geq 0$.\\
 Now for $u \in H^+$, setting $r=\frac{n}{1+q}$ and by the Sobolev embedding we obtain
\begin{align*}
\mathcal{J}_{\la,M} (u) &\geq \frac{b}{2n} \|u\|^n - \frac{\la (2n-1-q)}{2n(q+1)} H(u) \geq \frac{b}{2n} \|u\|^n - \frac{\la (2n-1-q)}{2n(q+1)} l \left(\int_{\Omega} |u|^n ~dx \right)^{1/r}\\
&= c_1 \|u\|^{n} - c_2 \|u\|^{q+1}
\end{align*} where $c_1 = \frac{b }{2n} $ and $c_2 = c_2(\lambda).$\\
We observe that the minimum of the function $g(x)= c_1 x^n - c_2 x^{q+1}$ is achieved at $x=\left(\frac{c_2(q+1)}{c_1 n}\right)^{\frac{1}{n-q-1}}.$
Therefore,
\begin{align*}
\inf_{u \in N_{\la,M}} \mathcal{J}_{\la,M} (u) \geq g \left(\frac{c_2(q+1)}{c_1 n}\right)^{\frac{1}{n-q-1}} = \left(\frac{c_2^n}{c_1^{q+1}}\right)^{\frac{1}{n-1-q}} \left( \left(\frac{q+1}{n}\right)^{\frac{n}{n-1-q}}- \left(\frac{q+1}{n}\right)^\frac{q+1}{n-1-q}\right).
\end{align*}

From this it follows that $$\theta \geq - C(q,n,b) \la^{\frac{n}{n-q-1}} $$
where $C(q,n,b)>0$. This completes the proof of Theorem \ref{bdd}.
\hfill{\QED}
\end{proof}

Now since $\mathcal{J}_{\la,M}$ is bounded below on $N_{\la,M}$, by the Ekeland variational principle we get a sequence $\{u_k\}_{k \in \mathbb{N}} \subset N_{\la, M}\backslash\{0\}$ such that
\begin{equation}\label{Ekeland}
     \left\{
         \begin{alignedat}{2}
             {}    \mathcal{J}_{\la,M}(u_k)
             & {} \leq \theta+ \frac{1}{k};
             \\
              \mathcal{J}_{\la,M}(v)
             & {}\geq \mathcal{J}_{\la,M}(u_k) -\frac{1}{k}\|u_k-v\|,
             && \ \ \ \forall v \in N_{\la, M}.
        \end{alignedat}
     \right.
\end{equation}
\begin{Lemma}\label{lemmaq}
There exists a constant $C_0>0$ such that $\theta \leq - C_0$.
\end{Lemma}
\begin{proof}
Let $u\in H^+$, then $\exists$ $t_1(u)>0$ such that $t_1 u \in N^+_{\la ,M}$ and $\psi_{u,M}(t_1)= \la H(u)$. In that case,
\begin{align*}
\mathcal{J}_{\la,M}(t_1 u)&= \frac{1}{n}\bigg(\frac{a}{2}\|t_1 u\|^{2n} + b\|t_1 u\|^n\bigg)-\frac{1}{2} \int_{\Omega} (|x|^{-\mu}*F(t_1 u))F(t_1 u)~dx- \frac{\la}{q+1} \int_{\Omega} h(x)|t_1 u|^{q+1}~dx\\
&= \frac{1}{n}\bigg(\frac{a}{2}\|t_1u\|^{2n} + b\|t_1u\|^n\bigg)-\frac{1}{2} \int_{\Omega} (|x|^{-\mu}*F(t_1u))F(t_1u)~dx\\
&\ \ \ \ -\frac{1}{q+1} \bigg( a\|t_1u\|^{2n}+b\|t_1u\|^n-\int_{\Omega}(|x|^{-\mu}*F(t_1u))f(t_1u)t_1u~dx \bigg).
\end{align*}
Since $\Phi^{'}_{u,M}(t_1)=0$, $\Phi^{''}_{u,M}(t_1)>0$ and from \eqref{KC21} we obtain
\begin{align*}
\mathcal{J}_{\la,M}(t_1 u) & = \frac{-(n-1-q)}{2n(q+1)} b \|t_1 u\|^n + \int_{\Omega}(|x|^{-\mu}*F(t_1u))\bigg(\frac{2n+q}{2n(q+1)} f(t_1u)t_1u \nonumber\\
&\ \ \ \ \  - \frac{1}{2} F(t_1u)-\frac{f'(t_1u)(tu)^2}{2n(q+1)}\bigg)~dx -\frac{1}{2n(q+1)} \int_{\Omega} (|x|^{-\mu}*f(t_1 u)t_1 u) f(t_1 u)t_1 u ~dx\nonumber\\
& \leq \frac{-(n-1-q)}{2n(q+1)} b \|t_1 u\|^n + \int_{\Omega}(|x|^{-\mu}*F(t_1u))\bigg(\frac{2n+q}{2n(q+1)} -\frac{p+2}{2n(q+1)}\nonumber\\
&\ \ \ \ \ \ \ - \frac{p+1}{2n(q+1)}\bigg)f(t_1u)t_1u ~dx- \frac{1}{2} \int_{\Omega}(|x|^{-\mu}*F(t_1u))F(t_1u)~dx.
\end{align*}
Since $q<n-1$ and $p+1 > 2n-1$ we set $2n+q-(2p+3) \leq 3n-1-(4n-1) <0$ and then $\theta \leq \inf_{u \in N_{\la,M}^+\cap H^+} \mathcal{J}_{\la,M}(u)\leq - C_0 <0.$
\hfill{\QED}\\
\end{proof}
Then by \eqref{Ekeland} and Lemma \ref{lemmaq}, we have for large $k$,\\\
\begin{equation}\label{KC9}
\mathcal{J}_{\la,M}(u_k) \leq - \frac{C_0}{2}.
\end{equation}
Also since $u_k \in N_{\la,M}\backslash\{0\}$ we have
\begin{align*}
\mathcal{J}_{\la,M}(u_k)&=a \|u_k\|^{2n}\bigg(\frac{p+2-2n}{2n(p+2)}\bigg)+ b \|u_k\|^{n}\bigg(\frac{p+2-n}{n(p+2)}\bigg)- \la\bigg(\frac{p+1-q}{(1+q)(p+2)}\bigg)H(u_k)\\
&\ \ \ \ \ \ \ \ -\frac{1}{2} \int_{\Omega} (|x|^{-\mu}*F(u_k))\left(F(u_k)-\frac{2f(u_k)u_k}{p+2}\right)~dx.
\end{align*}
then together with \eqref{KC9}, we have
$$- \la \bigg(\frac{p+1-q}{(1+q)(p+2)}\bigg)H(u_k) \leq - \frac{C_0}{2} \Longrightarrow H(u_k) \geq \frac{C_0(p+2)(1+q)}{2\la(p+1-q)} C_0 > 0$$
{\it i.e.}
\begin{equation}\label{sequence}
 H(u_k) > C >0\ \ \ \ \forall k \ \ \text{and}\ \ u_k \in N_{\la,M} \cap H^+.
\end{equation}

The following result shows that minimizers for $\mathcal{J}_{\la,M}$ in any of the subsets of $N_{\la,M}$ are critical points for $\mathcal{J}_{\la,M}.$
\begin{Lemma}
Let u be a local minimizer for $J_{\la,M}$ on any subsets of $N_{\la,M}$ such that $u \not\in N_{\la,M}^0$. Then $u$ is a critical point of $\mathcal{J}_{\la,M}.$
\end{Lemma}
\begin{proof}
Let $u$ be a local minimizer for $\mathcal{J}_{\la, M}.$ Then, in any case $u$ is a minimizer for $\mathcal{J}_{\la,M}$ under the constraint $I_{\la,M}(u):=\langle \mathcal{J}'_{\la,M}(u),u \rangle =0.$ Hence , by the theory of Lagrange multipliers , there exists a $\mu \in \mathbb{R}$ such that $\mathcal{J}'_{\la,M} = \mu I_{\la,M}'(u).$ Thus $\langle \mathcal{J}'_{\la,M}(u), u \rangle = \mu \langle I'_{\la,M}(u),u \rangle = \mu \Phi''_{\la, M}(1)=0,$ but $u \not\in N^0_{\la,M}$ and so $\Phi''_{\la, M}(1) \neq 0.$ Hence $\mu=0.$
\hfill{\QED}
\end{proof}

\begin{Lemma}\label{compl1}
Let $\la \in (0, \la_0)$ where $\la_0= \frac{bn}{(2n-1-q)l} M_0$. Then given any $u \in N_{\la, M} \backslash \{0\},$ then there exists $\epsilon >0$ and a differentiable function $\xi : B(0, \epsilon) \subset W_0^{1,n}(\Omega) \rightarrow \mathbb{R}$ such that $\xi(0)=1$, and $\xi(w)(u-w) \in N_{\la,M}$ and for all $w \in W_0^{1,n}(\Omega)$
\begin{align}\label{4.10}
\langle &\xi'(0),w \rangle = \frac{n(2a\|u\|^n+b) \int_{\Omega} |\nabla(u)|^{n-2} \nabla u.\nabla w~dx- \la (q+1) \int_{\Omega} h(x) |u|^{q-1} u w ~dx - \langle S(u),w \rangle} {a(2n-1-q)\|u\|^{2n}+ b(n-1-q)\|u\|^n + R(u)}
\end{align}
where $$R(u)=\int_{\Omega} (|x|^{-\mu}*F(u))(qf(u)-f'(u).u).u~dx -\int_{\Omega} (|x|^{-\mu}*f(u).u)f(u)u~dx$$
and $$\langle S(u),w \rangle= \int_{\Omega} (|x|^{-\mu}*F(u))(f'(u)u +f(u)) w ~dx + \int_{\Omega} (|x|^{-\mu}*f(u)u)f(u)w~dx .$$
\end{Lemma}
\begin{proof}
Fix $u \in N_{\la,M}\backslash \{0\},$ define a function $G_u: \mathbb{R}\times W_0^{1,n}(\Omega) \rightarrow \mathbb{R}$
\begin{align*}
G_u(t,v)&=a t^{2n-1-q}\|u-v\|^{2n}+ b t^{n-1-q} \|u-v\|^{n} - \frac{1}{t^q} \int_{\Omega} (|x|^{-\mu}*F(t(u-v)))f(t(u-v)).(u-v) ~dx \\
&-\la \int_{\Omega} h|u-v|^{q+1} ~dx.
\end{align*}

Then $G_u \in C^1(\mathbb{R}\times W_0^{1,n}(\Omega), \mathbb{R})$ and
$$G_u(1,0)=a \|u\|^{2n}+b \|u\|^n -\int_{\Omega} (|x|^{-\mu}*F(u))f(u).u~dx -\la \int_{\Omega} h |u|^{q+1}~dx = \Phi_u'(1)=0$$
and $$\frac{\partial}{\partial t} G_u(1,0)= a(2n-1-q)\|u\|^{2n} + b(n-1-q)\|u\|^{n} + q \int_{\Omega} (|x|^{-\mu}*F(u))f(u).u - B(u)= \phi_u''(1)\neq 0.$$
Then by the implicit function theorem, there exists $\epsilon >0$ and a differentiable function $\xi : B(0, \epsilon) \subset W_0^{1,n}(\Omega) \rightarrow \mathbb{R}$ such that $\xi(0)=1$ and $G_u(\xi(w),w)=0 \ \ \forall w \in B(0, \epsilon)$ which is equivalent to $\langle \mathcal{J}_{\la,M}'(\xi(w)(u-w),\xi(w)(u-w))\rangle =0 \ \ \forall v \in B(0, \epsilon)$. Thus, $\xi(w)(u-w) \in N_{\la,M}$ and differentiating
\begin{align*}
G_u(\xi(w),w)&= a (\xi(w))^{2n-1-q}\|u-w\|^{2n} + b (\xi(w))^{n-1-q} \|u-w\|^n \\
&-\frac{1}{(\xi(w))^q} \int_{\Omega} (|x|^{-\mu}*F(\xi(w))(u-w))f(\xi(w)(u-w))(u-w)-\la \int_{\Omega} h(x)|u-w|^{q+1}=0
\end{align*}
with respect to $w$, we obtain \eqref{4.10}.
%\begin{align}
%\langle &\xi'(0), \rangle = \frac{n(2a\|u\|^n+b) \int_{\Omega} |\nabla(u)|^{n-2} \nabla u.\nabla w- \la (q+1) \int_{\Omega} h(x) |u|^{q-1} u w - \langle S(u),w \rangle} {a(2n-1-q)\|u\|^{2n}+ b(n-1-q)\|u\|^n + R(u)}
%\end{align}
%where $$R(u)=\int_{\Omega} (|x|^{-\mu}*F(u))(qf(u)-f'(u).u).u -\int_{\Omega} (|x|^{-\mu}*f(u).u)f(u)u$$
%and $$\langle S(u), w\rangle= \int_{\Omega} (|x|^{-\mu}*F(u))(f'(u)u +f(u)) w + \int_{\Omega} (|x|^{-\mu}*f(u)u)f(u)w.$$
\hfill{\QED}\\
\end{proof}
Similarly we have:
\begin{Lemma}\label{compli2}
Let $\la \in (0, \la_0)$ where $\la_0= \frac{bn}{(2n-1-q)l} M_0$. Then there exists $u \in N^-_{\la,M} \backslash \{0\},$ then there exists $\epsilon >0$ and a differentiable function $\xi^- : B(0, \epsilon) \subset W_0^{1,n}(\Omega) \rightarrow \mathbb{R}$ such that $\xi^-(0)=1$, and $\xi^-(w)(u-w) \in N^-_{\la,M}$ and for all $w \in W_0^{1,n}(\Omega)$
\begin{align*}
\langle &(\xi^-)'(0),w \rangle = \frac{n(2a\|u\|^n+b) \int_{\Omega} |\nabla(u)|^{n-2} \nabla u.\nabla w ~dx- \la (q+1) \int_{\Omega} h(x) |u|^{q-1} u w ~dx - \langle S(u),w \rangle} {a(2n-1-q)\|u\|^{2n}+ b(n-1-q)\|u\|^n + R(u)}
\end{align*}
where $R(u)$ and $S(u)$ are as in lemma~\ref{compl1}.
%$$R(u)=\int_{\Omega} (|x|^{-\mu}*F(u))(qf(u)-f'(u).u).u -\int_{\Omega} (|x|^{-\mu}*f(u).u)f(u)u$$%
%and $$\langle S(u),w \rangle = \int_{\Omega} (|x|^{-\mu}*F(u))(f'(u)u +f(u)) w + \int_{\Omega} (|x|^{-\mu}*f(u)u)f(u)w. $$
\end{Lemma}
\begin{proof}
For any $u \in N_{\la,M}^-$, $\Phi^{'}_{u,M} (1)=0$ and $ \Phi^{''}_{u,M} (1)<0$.
%\begin{equation}
%\Phi^{'}_{u,M} (t) = t^{n-1} \|u\|^n m(\|tu\|^n) - \la t^{q} \int_{\Omega} h(x) |u|^{q+1}~dx - \int_{\Omega} (|x|^{-\mu}* F(tu)) f(tu) u ~dx =0
%\end{equation}
%and
%\begin{align*}
    %  \Phi^{''}_{u,M} (t)&= n t^{n-2} \|u\|^{2n} m(\|tu\|^n)+ (n-1) t^{n-2} \|u\|^n m(\|tu\|
    %  ^n) - \la q t^{q-1} \int_{\Omega} h(x) |u|^{q+1}~dx\\
    %  &- \int_{\Omega} (|x|^{-\mu} *f(tu).u)f(tu)u ~dx - \int_{\Omega} (|x|^{-\mu}* F(tu))f'(tu) u^2 ~dx
     % <0.
%\end{align*}
This implies $ u \in \Gamma\backslash \{0\}$.  Then by Lemma \ref{compl1} there exists $\epsilon >0$ and a differentiable function $\xi^-:B(0, \epsilon)\subset W_0^{1,n}(\Omega) \rightarrow \mathbb{R}$ such that $\xi^-(0)=1$, and $\xi^-(w)(u-w) \in N_{\la,M}$ for all $w \in B(0, \epsilon).$ Then by the continuity of $\mathcal{J}^{'}_{\la,M}$ and $\xi^-$ and by choosing $\epsilon$ small enough we have
\begin{align*}
      \Phi^{''}_{\xi^-(u)(u-w),M} (1)&= n \|\xi^-(u)(u-w)\|^{2n} m(\|\xi^-(u)(u-w)\|^n)+ (n-1)\|\xi^-(u)
      (u-w)\|^n m(\|tu\|^n) \\
      &- \la q \int_{\Omega} h(x) |\xi^-(u)(u-w)|^{q+1}~dx\\
      &- \int_{\Omega} (|x|^{-\mu} *f(\xi^-(u)(u-w)).
      \xi^-(u)(u-w))f(\xi^-(u)(u-w))\xi^-(u)(u-w) ~dx \\
      &- \int_{\Omega} (|x|^{-\mu}* F(\xi^-(u)(u-w)))f'(\xi^-(u)(u-w))(\xi^-(u)(u-w))^2 ~dx
      <0
\end{align*}
that implies $\xi^-(w)(u-w) \in N_{\la,M}^-.$
\hfill{\QED}
\end{proof}\\
Now we prove the following result:
\begin{Proposition}\label{j}
Let $\la \in (0, \la_0)$ where $\la_0= \frac{bn}{(2n-1-q)l} M_0$. Assume $u_k\in N_{\la,M}$ is satisfying \eqref{sequence}. Then $\|\mathcal{J}^{'}_{\la,M}(u_k)\|_* \rightarrow 0$ as $k \rightarrow \infty.$
\end{Proposition}
\begin{proof}
Step 1: $\liminf_{k \rightarrow \infty} \|u_k\| >0.$\\
\ \ We know that from \eqref{sequence} that for large $k$, $H(u_k)\geq C >0$, so by using H\"older inequality we obtain
$C < \ H(u_k) \leq C_1 \|u_k\|^{q+1}$.\\
Step 2: We claim that
$$\liminf_{k \rightarrow \infty} (2n-1-q)a \|u_k\|^{2n} + b(n-1-q) \|u_k\|^n +q \int_{\Omega} (|x|^{-\mu}*F(u_k))f(u_k)u_k ~dx- B(u_k) >0.$$
Without loss of generality, we can assume that $u_k \in N^+_{\la,M}$ (if not replace $u_k$ by $t_1(u_k)u_k$). Arguing by contradiction, suppose that there exists a subsequence of $\{u_k\}$, still denoted by $\{u_k\}$, such that
$$0\leq (2n-1-q)a \|u_k\|^{2n} + b(n-1-q) \|u_k\|^n +q \int_{\Omega} (|x|^{-\mu}*F(u_k))f(u_k)u_k ~dx - B(u_k)= o_k(1).$$
From Step $1$ and the above equation we obtain that $\liminf_{k \rightarrow \infty} B(u_k) >0$ and $(2n-1-q)a \|u_k\|^{2n} + b(n-1-q) \|u_k\|^n \leq B(u_k)$ {\it i.e.} $u_k \in \Gamma\backslash \{0\} \ $ for all large $k.$\\
Since $u_k \in N^+_{\la,M}\backslash \{0\}$
\begin{align*}
-nb\|u_k\|^n+ \la (2n-1-q) H(u) +(2n-1) \int_{\Omega} (|x|^{-\mu}*F(u_k))f(u_k)u_k ~dx - B(u_k)&= o_k(1)
\end{align*}
which is a contradiction since $\Gamma_0 >0.$\\
Step 3: $\|\mathcal{J}^{'}_{\la,M}(u_k)\|_* \rightarrow 0$ as $k \rightarrow \infty$.\\
By using Lemma \ref{compl1}, there exists a differentiable function $\xi_k: B(0, \epsilon_k) \rightarrow \mathbb{R}$ for some $\epsilon_k >0$ such that $\xi_k(0)=1$ and $\xi_k(w)(u_k-w) \in N_{\la, M}\ \ \forall w \in B(0,\epsilon_k).$ Choose $0 < \rho < \epsilon_k$ and $f \in W_0^{1,n}(\Omega)$ such that $\|f\|=1.$ Let $w_\rho = \rho f.$ Then $\|w\|_{\rho}= \rho < \epsilon_k$ and define $\eta_{\rho}= \xi_k(w_{\rho})(u_k- w_{\rho})$. Then from the Taylor expansion and \eqref{kc3}, we obtain
\begin{equation}\label{qwer}
\begin{split}
\frac{1}{k} \|\eta_{\rho} - u_k\| &\geq \mathcal{J}_{\la ,M}(u_k) - \mathcal{J}_{\la ,M}(\eta_\rho) = \langle \mathcal{J}^{'}_{\la ,M}(\eta_\rho), u_k - \eta_{\rho} \rangle + o(\|u_k -\eta_\rho\|) \\
&= (1-\xi_k(w_{\rho}))\langle \mathcal{J}^{'}_{\la ,M}(\eta_\rho), u_k\rangle + \rho \xi_k(w_\rho) \langle \mathcal{J}^{'}_{\la ,M}(\eta_{\rho}),f\rangle + o(\|u_k -\eta_\rho\|).
\end{split}
\end{equation}
We also infer
\begin{align*}
\frac{1}{\rho} \|\eta_{\rho}- u_k\| &=\|\frac{(\xi_k(w_\rho)-1)}{\rho}u_k - \xi_k(w_\rho) f\| \rightarrow \|u_k \langle \xi^{'}_k(0),f \rangle -f \| \ \ \text{as} \ \ \rho \rightarrow 0.
\end{align*}
Since $u_k\in N_{\lambda,M}$, we have also $\frac{1- \xi_k(w_{\rho})}{\rho} \langle \mathcal{J}^{'}_{\la ,M}(\eta_\rho), u_k  \rangle \rightarrow 0 \ \ \text{as} \ \rho \rightarrow 0.$\\
Thus, dividing the expression in \eqref{qwer} by $\rho$ and doing $\rho\to 0^+$, we get
$$\langle \mathcal{J}^{'}_{\la ,M}(u_k), f \rangle \leq \frac{\|f\|}{k} (\|u_k\| \|\xi^{'}_k(0)\|_*+O(1))$$ which implies that $$\|\mathcal{J}^{'}_{\la,M}(u_k)\|_* \rightarrow 0 \ \ \text{as} \ \ k \rightarrow \infty$$
if $\|\xi^{'}_k(0)\|_*$ is bounded uniformly in $k$. To prove that, using \eqref{imp} and the boundedness of the sequence $(u_k)$ in $W^{1,n}_0(\Omega)$, we only need to show that for any $f\in W^{1,n}_0(\Omega)$, $\langle S(u_k),f\rangle$ is uniformly bounded in $k$.
For the subcritical case, {\it i.e.} $\beta \in (0,\frac{n}{n-1})$, it holds since for any $\epsilon>0$ and $q>1$, there exists $C_{\epsilon,q, \beta}>0$ such that
\begin{equation*}
\displaystyle\exp(q|t|^{\beta}) \leq C_{\epsilon,q,\beta}\exp(\epsilon|t|^{\frac{n}{n-1}}), \quad \forall t\in \R.
\end{equation*}
Then by Theorem \ref{TM-ineq} we obtain $\langle S(u_k), f \rangle \leq C \|f\|$ with $C>0$ independent of $k$. Consider  now the critical case, {\it i.e.} $\beta = \frac{n}{n-1}$. From the boundedness of $R(u_k)$ (see statement of Lemma~\ref{compl1}), it follows that
\begin{equation*}
\displaystyle\sup_k\int_{\Om}(|x|^{-\mu}*F(u_k))f(u_k)u_k~dx<\infty,
\end{equation*}
\begin{equation*}
\displaystyle\sup_k\int_{\Om}(|x|^{-\mu}*F(u_k))f'(u_k)u_k^2~dx<\infty
\end{equation*}
and
\begin{equation*}
\displaystyle\sup_k\int_{\Om}(|x|^{-\mu}*f(u_k)u_k)f(u_k)u_k~dx<\infty.
\end{equation*}
Then for any $\phi\in C^\infty_c(\Omega)$, we have by Vitali's convergence theorem and up to a subsequence
\begin{equation}\label{convergence-S}
\displaystyle\langle S(u_k),\phi\rangle\to \langle S(u_0),\phi\rangle
\end{equation}
where $u_0$ is the weak limit of $(u_k)_{k\in \N}$ in $W^{1,n}_0(\Omega)$. From \eqref{convergence-S}, we have that there exists $C>0$ independent of $k$ such that
\begin{equation}\label{conv-S-sob}
\displaystyle|\langle S(u_k),\phi\rangle|\leq C \|\phi\|.
\end{equation}
Using a density argument, we conclude that \eqref{conv-S-sob} holds for any $\phi\in W^{1,n}_0(\Omega)$. This completes the proof in the critical case.

\hfill{\QED}
\end{proof}
\begin{Theorem}\label{exis1}
Let $\beta < \frac{n}{n-1}$ and let $\la \in (0, \la_0)$ where $\la_0= \frac{bn}{(2n-1-q)l} M_0$. Then there exists a weak solution to $({\mathcal P}_{\lambda,M})$ $u_{\la} \in N^+_{\la,M} \cap H^+$ such that $\mathcal{J}_{\la, M}(u_\la)= \inf_{u\in N_{\la, M}\backslash \{0\}} \mathcal{J}_{\la, M}(u).$
\end{Theorem}
\begin{proof}
Let $u_k$ be a minimizing sequence satisfying $\mathcal{J}_{\la, M}(u_k) \rightarrow \theta$ as $k \rightarrow \infty$ and $\mathcal{J}_{\la, M}(v) \geq \mathcal{J}_{\la, M}(u_k) - \frac{1}{k} \|u_k-v\|,   \ \forall v \in N_{\la}.$ Then by using Proposition \ref{j} we obtain $\{u_k\}$ is $(\text{PS})_{\theta}$ sequence. Then from Lemma \ref{kc-PS-bdd}
we get $\{u_k\}$ is a bounded sequence in $W_0^{1,n}(\Omega).$ Also there exists a subsequence of $\{u_k\}$ (denoted by same sequence) and $u_{\la}$ such that $u_k \rightharpoonup u_\la $ weakly in $W_0^{1,n}(\Omega)$ and $u_k \rightarrow u_{\la}$ strongly in $L^r(\Omega)\ $ for $ r \geq 1$ and $u_k \rightarrow u_\la$  a.e. in $\Omega.$ Then using $f(t) \leq C_{\epsilon,\beta} \exp(\epsilon t^{\frac{n}{n-1}})$ for $\epsilon>0$ small enough and  from Theorem \ref{TM-ineq},
we obtain that $f(u_k) $ and $(|x|^{-\mu}*F(u_k))$ are uniformly bounded in $L^q(\Omega)$ for all $q>1.$ Then again by Vitali's convergence theorem, we obtain
\begin{align*}
\left|\int_{\Omega} (|x|^{-\mu}*F(u_k))f(u_k) (u_k- u_\la)~dx\right| \rightarrow 0 \ \text{as} \ \ k \rightarrow \infty.
\end{align*}
and by Proposition \ref{j}, we have $\mathcal{J}^{'}_{\la, M}(u_k- u_\la) \rightarrow 0.$ Then we conclude that
\begin{equation}\label{KC11}
m(\|u_k\|^n) \int_{\Omega} |\nabla u_k|^{n-2} \nabla u_k. \nabla(u_k- u_{\la}) ~dx \rightarrow 0
\end{equation}
On the other hand, using $u_k \rightarrow u_\la$ weakly and by boundedness of $m(\|u_k\|^n)$ we have
\begin{equation}\label{KC12}
m(\|u_k\|^n) \int_{\Omega} |\nabla u_\la|^{n-2} \nabla u_\la. \nabla(u_k- u_{\la}) ~dx \rightarrow 0.
\end{equation}
Substracting \eqref{KC12} from \eqref{KC11}, we get,
\begin{equation*}
m(\|u_k\|^n) \int_{\Omega} (|\nabla u_k|^{n-2} \nabla u_k- |\nabla u_\la|^{n-2} \nabla u_\la). \nabla(u_k- u_{\la}) ~dx \rightarrow 0.
\end{equation*}
Now by using this and following inequality,
\begin{equation*}
|a_1-a_2|^n \leq 2^{n-2} (|a|_1^{n-2} a_1 -|a_2|^{n-2} a_2)(a_1-a_2) \ \text{for all} \ \ a_1, a_2\in \mathbb{R}^n
\end{equation*}
with $a_1= \nabla u_k$ and $a_2= \nabla u_\la$, we obtain
\begin{equation*}
m(\|u_k\|^n) \int_{\Omega} |\nabla u_k - \nabla u_\la|^n ~dx \rightarrow 0 \ \ \text{as} \ k \rightarrow  \infty.
\end{equation*}
Since $m(t) \geq b$, then we obtain $u_k \rightarrow u_{\la}$ strongly in $W_0^{1,n}(\Omega)$ and by Lemma \ref{kc-ws}
\begin{equation*}
\int_{\Omega} (|x|^{-\mu}*F(u_k))f(u_k) \phi ~dx \rightarrow \int_{\Omega} (|x|^{\mu}*F(u_\la))f(u_\la) \phi ~dx
\end{equation*}
and also
\begin{equation*}
\int_{\Omega} h(x) u_k^{q-1} u_k \phi ~dx \rightarrow \int_{\Omega} h(x) u_\la^{q-1} u_\la \phi ~dx
\end{equation*}
for all $\phi \in W_0^{1,n}(\Omega)$.
Therefore, $u_\la$ satisfies $(\mathcal{P}_{\la,M})$ in weak sense and hence $u_k \in N_{\la, M}.$ Moreover, $\theta \leq \mathcal{J}_{\la,M}(u_\la) \leq \liminf_{k \rightarrow \infty} \mathcal{J}_{\la, M}(u_k)= \theta.$ Hence $u_\la$ is a minimizer for $\mathcal{J}_{\la, M}$ in $N_{\la, M}$.\\
Using \eqref{sequence}, we have $\int_{\Omega} h(x) |u_{\la}|^{q+1} >0$, then there exists $t_1(u_\la)$ such that $t_1(u_\la) u_\la \in N^{+}_{\la, M}$. We now claim that $t_1(u_\la)=1$ {\it i.e.} $u_\la \in N^+_{\la, M}.$ Suppose that $t_1(u_\la)<1$ and then $t_2(u_\la)=1$ and $u_\la \in N^{-}_{\la, M}.$ Now $\mathcal{J}_{\la,M}(t_1(u_\la)u_\la) < \mathcal{J}_{\la, M}(u_\la) \leq \theta $ which yields a contradiction, since $t_1(u_{\la}) u_\la \in N_{\la,M}.$ Thus, $u_\lambda$ is non negative and nontrivial. From the strong comparison principle, we get $u_\lambda>0$ in $\Om$.
\hfill{\QED}
\end{proof}
\begin{Theorem}\label{exis2}
Let $\beta < \frac{n}{n-1}$ and let $\la \in (0, \la_0)$ where $\la_0= \frac{bn}{(2n-1-q)l} M_0$. Then $u_{\la} \in N^+_{\la,M} \cap H^+$ is a non-negative local minimum for $\mathcal{J}_{\la, M}(u_\la)$ in $W_0^{1,n}(\Omega).$
\end{Theorem}
\begin{proof}
Since $u_\la \in N^+_{\la,M} \cap H^+$ then we have a $t_*(u_\la)$ such that $1=t_1(u_\la) < t_*(u_\la).$ Hence by the continuity of $u \rightarrow t_*(u),$ given $\epsilon >0$ there exists $\delta_{\epsilon} >0$ such that $$(1+\epsilon) <t_*(u_{\la}-w) \ \ \ \text{for all} \ \|w\| < \delta_{\epsilon}$$ and from Lemma \ref{compl1} we have, for $\delta>0$ small enough, a continuously differentiable map $t: B(0, \delta) \rightarrow \mathbb{R}^+$ such that $t(w)(u_{\la}-w) \in N_{\la, M}, t(0)=1.$ Then we have $$t_1(u_{\la}-w)=t(w)< 1+\epsilon < t_*(u_{\la}-w)$$ for $\delta$ small enough. Since $t_*(u_{\la}-w) >1$ for all $\|w\| < \delta ,$ we obtain  $$\mathcal{J}_{\la, M}(u_\la) \leq \mathcal{J}_{\la, M}(t_1(u_\la-w) (u_\la-w)) \leq \mathcal{J}_{\la, M}(u_\la -w),\; \mbox{if }\|w\| < \delta $$ which implies that $u_\la$ is a local minimizer for $\mathcal{J}_{\la, M}.$

%Now we show that $u_\la$ is non-negative local minimum for $\mathcal{J}_{\la,M}$ on $W_0^{1,n}(\Omega).$ Suppose $u_\la \ngeq 0$ then we can take $\tilde{u}_\la = t_1(|u_\la|) |u_\la|$ which is a non-negative function in $N_{\la, M} \cap H^+$ and $\psi_{u_\la}(t)= \psi_{|u_\la|}(t)$, $t_*(u_\la)= t_*(|u_\la|)$ and $1\leq t_1(u_\la)\leq t_1(|u_\la|).$ Also $|u_\la| \in H^+$ then we have $\mathcal{J}_{\la,M}(\tilde{u_\la}) \leq \mathcal{J}_{\la, M}(|u_\la|) \leq \mathcal{J}_{\la,M}(u_\la).$ Hence $\tilde{u_\la}$ minimizes $\mathcal{J}_{\la, M}$ on $N_{\la, M}\backslash \{0\}.$ Then follow the same steps in the upper paragraph to show that $\tilde{u_\la}$ is a local minimum for $\mathcal{J}_{\la, M}$ on $W_0^{1,n}(\Omega)$ which implies that $u_\la$ is non-negative local minimum of $\mathcal{J}_{\la, M}$ on $W_0^{1,n}(\Omega).$
\hfill{\QED}
\end{proof}
\begin{Theorem}\label{exis3}
Let $\beta < \frac{n}{n-1}$ and let $\la \in (0, \la_0)$ where $\la_0= \frac{bn}{(2n-1-q)l} M_0$. Then $\mathcal{J}_{\la, M}$ achieve its minimizers on $N^-_{\la, M}.$
\end{Theorem}
\begin{proof}
Let $u \in N^-_{\la,M}$. Then
\begin{align*}
(2n-1) a \|u\|^{2n} + (n-1)b \|u\|^n- \la qH(u)- \left(\int_{\Omega} (|x|^{-\mu}*f(u)u)f(u).u + \int_{\Omega} (|x|^{-\mu}*F(u)) f'(u) u^2\right) <0.
\end{align*}
Then \eqref{KC2} implies that
\begin{equation}\label{4.39}
\begin{split}
(2n-1-q) a \|u\|^{2n} +(n-1-q) &b \|u\|^n +q \int_{\Omega} (|x|^{-\mu}*F(u)) f(u).u \\
&- \left(\int_{\Omega} (|x|^{-\mu}*f(u)u)f(u).u + \int_{\Omega} (|x|^{-\mu}*F(u)) f'(u) u^2\right) <0.
\end{split}
\end{equation}
Using $p+1 >2n$ it is easy to deduce from  \eqref{4.39} that $\exists\ c>0,\ \|u\| \geq c>0$ for any $u \in N^-_{\la,M}\backslash \{0\}$ from which it follows that $N^-_{\la,M}\backslash \{0\}$ is a closed set. Also as in Lemma \ref{inf} we can prove that $N_{\la, M}^- \subset \Gamma$ and then $\inf_{u \in N^-_{\la, M}\backslash \{0\}} B(u) \geq \tilde{c} >0.$ Therefore, for $\la < \la_0$ small enough,
\begin{equation}\label{xi-argument}
\displaystyle\inf_{u\in N^-_{\la,M}\backslash \{0\}}B(u)+nb\|u\|-(2n-1-q)\lambda H(u)-(2n-1)\int_\Om |x|^{-\mu}*F(u)f(u)u>0.
\end{equation}
Now taking $\theta^- = \min_{u \in N^-_{\la, M}\backslash \{0\}} \mathcal{J}_{\la, M}(u) > - \infty.$ From Ekeland variational principle, there exist $\{v_k\}_{k\in \mathbb{N}}$ a non-negative minimizing sequence such that
$$\mathcal{J}_{\la, M}(v_k) \leq \inf_{u \in N^-_{\la, M}} \mathcal{J}_{\la,M}(u) + \frac{1}{k} \ \text{and} \ \mathcal{J}_{\la, M}(u) \geq \mathcal{J}_{\la, M}(v_k) - \frac{1}{k}\|v_k-u\|\ \ \forall \ u \in N^-_{\la, M}.$$ From $\mathcal{J}_{\la,M}(v_k) \to \theta^-$ as $k\to\infty$ and $v_k \in N_{\la, M}$, it is easy to prove  that $\|v_k\| \leq C$ (as in Lemma \ref{kc-PS-bdd}).
Indeed,
\begin{align*}
\left| a\|v_k\|^{2n}+b \|v_k\|^n -\lambda H(v_k)- \int_{\Omega}(|x|^{-\mu}* F(v_k))f(v_k)v_k ~dx\right| =o(\|v_k\|)
\end{align*}
and
\begin{align*}
& C+o(\|v_k\|)\geq \mathcal{J}_{\la,M}(v_k)- \frac{1}{2n} \langle \mathcal{J}^{'}_{\la, M}(v_k),v_k \rangle
\geq \frac{b}{2n} \|v_k\|^{2n} -C(\lambda)\|v_k\|^{q+1}
\end{align*}
imply $\|v_k\| \leq C.$ Thus we get $\|S(v_k)\|_* \leq C_1$ and from \eqref{xi-argument} we have $\|\xi_k^-(0)\|_* \leq C_2.$ Now the rest of the proof can be done as in the proof of Theorem \ref{exis1} with the help of Lemma~\ref{compli2}.
\hfill{\QED}
\end{proof}\\ \\
{\it Proof of Theorem \ref{first} for $\beta \in \bigg(1, \frac{n}{n-1}\bigg)$ }: The proof follows from Lemma \ref{exis1} and Theorem \ref{exis2}.
\hfill{\QED}

Now we establish the following compactness result in the critical case.
\begin{Lemma}\label{compactness}
There exists $C=C(p,q,n)>0$ such that for any $(u_k)_{k\in\N} \subset W_0^{1,n}(\Omega)$ satisfying
$$\mathcal{J}^{'}_{\la, M}(u_k) \rightarrow 0\ \ \text{and}\ \ \mathcal{J}_{\la, M}(u_k) \rightarrow c \leq \frac{m_0}{2n} \left(\frac{2n-\mu}{2n} \alpha_n\right)^{n-1}- C \la^{\frac{2(p+2)}{2p+3-q}} \quad\mbox{as }k\to\infty$$
is relatively compact in $W_0^{1,n}(\Omega).$
\end{Lemma}
\begin{proof}
As in Lemma \ref{kc-PS-bdd} we can prove that $(u_k)_{k\in\N}$ is bounded in $W_0^{1,n}(\Omega)$ and up to a subsequence $u_k \rightarrow u $ in $L^{\alpha}(\Omega)$ for all $\alpha \geq 1$, $u_k(x) \rightarrow u $ a.e in $\Omega$ and $\nabla u_k \rightarrow \nabla u$ a.e. in $\Omega$ and weakly in $L^n(\Omega).$
Also still up to a subsequence, there exist radon measures $\nu_1$ and $\nu_2$ such that $|\nabla u_k|^n \to \nu_1$ and $(|x|^{-\mu}*F(u_k))f(u_k)u_k \to \nu_2$ weakly as $k\rightarrow \infty$.\\
Let $B=\{x \in \overline{\Omega}: \exists \ r= r_x>0, \nu_1(B_r \cap \Omega) < \left(\frac{2n-\mu}{2n} \alpha_n\right)^{n-1}\}$ and let $A=\overline{\Omega}\backslash B.$ Then by Lemma \ref{wk-sol} we can infer that $A$ is a finite set, say $\{x_1, x_2, \dots, x_n\}.$ Since $\mathcal{J}^{'}_{\la, M}(u_k) \rightarrow 0$ and since $(u_k)_{k\in\N}$ is bounded in $W^{1,n}_0(\Om)$, we have that $\forall \ \phi \in C_c^{\infty}(\Omega)$,
\begin{equation}\label{wksol30}
\begin{split}
 0= \lim_{k \to \infty} \langle \mathcal{J}^{'}_{\la, M}(u_k), \phi \rangle &= \lim_{k \to \infty}\bigg[ m(\|u_k\|^n) \int_{\Omega} |\nabla u_k|^{n-2} \nabla u_k. \nabla \phi ~dx-\la \int_{\Omega} h(x) |u_k|^{q-1} u_k \phi ~dx \\
&+ \int_{\Omega} (|x|^{-\mu}*F(u_k)) f(u_k) \phi ~dx\bigg],
\end{split}
\end{equation}
\begin{equation}\label{wksol31}
\begin{split}
 0= \lim_{k \to \infty} \langle \mathcal{J}^{'}_{\la, M}(u_k), u_k \phi \rangle &= \lim_{k \to \infty} \bigg[m(\|u_k\|^n) \int_{\Omega} (|\nabla u_k|^{n-2} \nabla u_k. \nabla \phi u_k ~dx + |\nabla u_k|^n \phi) -\la \int_{\Omega} h(x) |u_k|^{q+1} \phi ~dx \\
 &+ \int_{\Omega} (|x|^{-\mu}*F(u_k)) f(u_k) u_k \phi ~dx \bigg],
\end{split}
\end{equation}
\begin{equation}\label{wksol32}
\begin{split}
 0= \lim_{k \to \infty} &\langle \mathcal{J}^{'}_{\la, M}(u_k), u \phi \rangle = \lim_{k \to \infty} \bigg[m(\|u_k\|^n) \int_{\Omega}( |\nabla u_k|^{n-2} \nabla u_k. \nabla \phi u ~dx +  |\nabla u_k|^{n-2} \nabla u_k. \nabla u  \phi ~dx) \\
& +  \int_{\Omega} (|x|^{-\mu}*F(u_k)) f(u_k) u \phi ~dx \bigg] - \la \int_{\Omega} h(x) |u|^{q} \phi ~dx.
\end{split}
\end{equation}
Substituting \eqref{wksol32} in \eqref{wksol31} and taking into account \eqref{wksol30}, we get $\forall \ \phi \in C_c^{\infty}(\Omega)$
\begin{equation}\label{wksol35}
\begin{split}
\int_{\Omega} (|x|^{-\mu}*F(u_k)) f(u_k) u_k \phi &=\lim_{k \to \infty} m(\|u_k\|^n) \int_{\Omega} |\nabla u_k|^n - \nabla u_k|^{n-2} \nabla u_k. \nabla u  \phi ~dx\\
&+ \int_{\Omega} (|x|^{-\mu}*F(u)) f(u) u \phi ~dx +o_k(1).
\end{split}
\end{equation}
Now we take the cut-off function $\psi_\delta \in C_c^{\infty}(\Omega)$ such that $\psi_{\delta} = 1$ in $B_{\delta}(x_j)$ $\forall \ j=\{1, \dots , m\}$ and $\psi_\delta(x)=0$ in $B_{\delta}^{c}(x_j)$ with $|\psi_{\delta}| \leq 1.$ Then by taking $\phi= \psi_{\delta}$ in \eqref{wksol35} and since
\begin{equation*}
\begin{split}
0 &\leq \left| \int_{\Omega} (|\nabla u_k| \nabla u_k. \nabla u) \psi_{\delta} ~dx \right| \leq \int_{\Omega} |\nabla u_k|^{n-1}| \nabla u| |\psi_{\delta}|~dx \\
&\leq \int_{B_{2 \delta}} |\nabla u_k|^{n-1} |\nabla u| ~dx \leq \left(\int_{\Omega} |\nabla u_k|^n~dx\right)^{n/(n-1)}\left( \int_{\cup_j B_{2 \delta} (x_j)}|\nabla u|^n ~dx \right)^{1/n} \to 0 \ \text{as} \ \delta \to 0,
\end{split}
\end{equation*}
we deduce after letting $\delta\to 0$ that
\begin{equation}\label{wksol36}
\nu_2(A) \geq m_0 \nu_1(A) \geq m_0 \left(\frac{2n-\mu}{2n} \alpha \right)^{n-1}.
\end{equation}
On the other hand, by using the same argument as in Lemma \ref{wk-sol}\ (in particular see \ref{wk-sol2})
we can prove that for any compact set $K \subset \Omega_{\delta}= \Omega \backslash \cup_{i=1}^n B_{\delta}(x_i)$
\begin{equation*}
\lim_{k \to \infty} \int_{K} (|x|^{-\mu}* F(u_k)) f(u_k) u_k ~dx = \int_{K} (|x|^{-\mu}* F(u)) f(u)u ~dx.
\end{equation*}
Thus, we obtain
\begin{equation*}
\begin{split}
nc&=\lim_{k \to \infty} n \ \mathcal{J}_{\la, M}(u_k)- \frac{1}{2} \langle \mathcal{J}^{'}_{\la, M} (u_k), u_k \rangle =  \lim_{k \to \infty} \left( M(\|u_k\|^n)- \frac{1}{2} m(\|u_k\|^n) \|u_k\|^n \right)\\
&+ \lim_{k \to \infty} \frac{1}{2}\int_{\Omega} (|x|^{-\mu}*F(u_k))(f(u_k)u_k- n F(u_k))~dx+ \la \left(\frac{1}{2}- \frac{n}{q+1}\right) \int_{\Omega} h(x) |u_k|^{q+1}~dx.
\end{split}
\end{equation*}
Since $$ \int_{\Omega} (|x|^{-\mu}*F(u_k)) F(u_k)~dx \to \int_{\Omega} (|x|^{-\mu}*F(u))F(u)~dx,$$
 $$\frac{1}{2}\int_{\Omega} (|x|^{-\mu}*F(u_k))f(u_k) u_k ~dx \to \frac{1}{2}\int_{\Omega} (|x|^{-\mu}*F(u))f(u) u ~dx+ \frac{\nu_2(A)}{2},$$
together with \eqref{wksol36} it follows that
\begin{align*}
nc \geq \frac{m_0}{2} &\left(\frac{2n-\mu}{2n} \alpha_n\right)^{n-1} + \la \left(\frac{1}{2}- \frac{n}{q+1}\right) \int_{\Omega} h(x) u^{q+1}~dx -\frac{n}{2} \int_{\Omega} (|x|^{-\mu}*F(u)) F(u)~dx\\
&+ \frac{1}{2} \int_{\Omega} (|x|^{-\mu}*F(u)) f(u)u ~dx.
\end{align*}
Consequently,
\begin{align*}
c &\geq \frac{m_0}{2n} \left(\frac{2n-\mu}{2n}\alpha \right)^{n-1} + \la \left(\frac{1}{2n}-\frac{1}{(q+1)}\right) \int_{\Omega} h u^{q+1}~dx + \left(\frac{1}{2n}- \frac{1}{2(p+1)}\right) \int_{\Omega} (|x|^{-\mu}*F(u))f(u)u ~dx\\
&\geq \frac{m_0}{2n} \left(\frac{2n-\mu}{2n}\alpha\right)^{n-1} - \|h\|_{L^{r'}(\Omega)} \la \left(\frac{2n-1-q}{2n (q+1)}\right) \left(\int_{\Omega} u^{(p+2)\frac{2n}{2n-\mu}}~dx \right)^{\frac{(q+1)(2n-\mu)}{2n(p+2)}} \\
&+ \frac{2p+2-2n}{2n(2p+2)(p+2)} \left(\int_{\Omega} u^{(p+2)\frac{2n}{2n-\mu}}~dx \right)^{\frac{2n-\mu}{n}}
\geq \frac{m_0}{2n} \left(\frac{2n-\mu}{2n}\alpha\right)^{n-1} - \inf_{t \in \mathbb{R}^+} \rho(t)
\end{align*}
where $r'= \left(1- \frac{(q+1)(2n-\mu)}{2n(p+2)}\right)^{-1}$ and $\rho(t)= \|h\|_{L^{r'}(\Omega)} \la \left(\frac{2n-1-q}{2n (q+1)}\right)t^{\frac{q+1}{2(p+2)}} - \frac{2p+2-2n}{2n(2p+2)(p+2)} t .$
Thus $c \geq \frac{m_0}{2n} \left(\frac{2n-\mu}{2n}{\alpha_n} \right)^{n-1}- \tilde{C} \la^{\frac{2(p+2)}{2p+3-q}}$ which completes the proof.
\hfill{\QED}\\
\end{proof}
%Let $\la_1= max\{\la: \theta_{\la,M} \leq \frac{m_0}{2n} \left(\frac{2n-\mu}{2n}\alpha \right)^{n-1}- \tilde{C} \la^{\frac{2(p+2)}{2p+3-q}}\}$ where C is an in above lemma.
Now we prove Theorem \ref{second} which concerns the critical case $\beta= \frac{n}{n-1}$.\\ 
{\it Proof of Theorem \ref{second}}
Let $u_k$ be a nonnegative minimizing sequence for $\mathcal{J}_{\la, M}$ on $N_{\la, M} \backslash \{0\}$ satisfying \eqref{Ekeland} then $u_k$ is bounded in $W^{1,n}_0(\Om)$. Using Proposition \ref{j} we get $u_k$ is $PS_{\theta}$ sequence with $\theta < \frac{m_0}{2n} \left(\frac{2n-\mu}{2n}{\alpha_n}  \right)^{n-1}- \tilde{C} \la^{\frac{2(p+2)}{2p+3-q}}$. Taking $\la$ small enough, using Lemma~\ref{lemmaq} and Lemma~\ref{compactness}, $\{u_k\}$ admits a strongly convergent subsequence. Let $u \in W_0^{1,n}(\Omega)$ be the limit of this  subsequence. Then arguing as in the proof of Theorems~\ref{exis2} and \ref{exis3}, we prove that $u$ is a non-trivial weak solution and $\mathcal{J}_{\la,M}(u)= \theta.$ By elliptic regularity and strong maximum principle, we infer that $u > 0$ in $\Omega.$ This completes the proof of Theorem~\ref{second}.
\hfill{\QED}

\end{document}